\documentclass[10pt,leqno]{amsart}
\usepackage{a4wide}
\usepackage{amssymb,upgreek}
\usepackage{amsmath}
\usepackage{amsthm}
\usepackage{mathrsfs,fancybox,mdframed}
\usepackage{bm,euscript,csquotes,comment,appendix,multirow}
\usepackage{upgreek}
\usepackage[curve]{xypic}
\usepackage[usenames,dvipsnames]{xcolor}
\usepackage{enotez}
\usepackage{calligra,yfonts}
\usepackage{multirow}
\usepackage{tikz-cd}
\usepackage{enumerate}

\usepackage{euscript}
\usepackage{hyperref}

\hypersetup{colorlinks,  citecolor=ForestGreen,  linkcolor=blue, urlcolor=black}

\DeclareMathOperator{\cHom}{\mathscr{H}\text{\kern -3pt {\calligra\large om}}\,}
\def\eZ{{Z}}
\def\la{\langle}
\def\ra{\rangle}
\def\aff{{\rm{aff}}}
\def\Waff{W_{\rm\aff}}
\def\Saff{S_{\rm\aff}}

\def\R{\mathbf{R}}
\def\a{{ \mathbf a}}
\def\b{{ \mathbf b}}
\def\q{{ \mathbf q}}
\theoremstyle{plain}
\def\k{{{k}}}
\def\tr{{\rm{tr}}}

\def\Z{\mathbb{Z}}
\def\cX{{\check X}}
\def\cPhi{{\check \Phi}}
\def\rk{\rm{rk}}
\def\WFf{{W^0_F}}

\def\WF{{W_F}}

\def\WC{{W_C}}
\def\A{\mathbf A}
\def\Bi{\mathbf {B.i}}
\def\Bii{\mathbf {B.ii}}
\newcommand{\cV}{\underline{\underline{\mathbf{M}}}}
\newcommand{\CH}{\underline{\underline{{{H_{\a,\b}}}}}}
\newcommand{\Hh}{{H_{\a,\b}}}
\newcommand{\Hn}{{H}}
\newcommand{\HF}{{{H}_F}}

\newcommand{\HC}{{H}_C}

\newcommand{\HGO}{{H}_{G_0}}
\newcommand{\id}{\operatorname{id}}
\newcommand{\im}{\operatorname{im}}
\newcommand{\gr}{\operatorname{gr}}

\newcommand{\Res}{\operatorname{Res}}
\newcommand{\Hom}{\operatorname{Hom}}
\newcommand{\RHom}{\operatorname{RHom}}
\newcommand{\End}{\operatorname{End}}
\newcommand{\Ext}{\operatorname{Ext}}

\newcommand{\Spec}{\operatorname{Spec}}

\newcommand{\coker}{\operatorname{coker}}

\newcommand{\Rees}{\textnormal{Rees}}

\def\D{\EuScript{D}}

\def\V{\EuScript{M}}

\def\Aa{\mathscr{A}}

\def\Ff{\mathscr F}
\def\loccit{\emph{loc. cit. }}

\def\z{\mathfrak{z}}
\def\r{{R}}

\def\Hhz{{\Hn\otimes_{\z } \Hn^o}}
\def\Hhzo{{\Hn^o\otimes_{\z} \Hn}}
\def\Hz{{\Hn\otimes_{\z } \Hn}}

\def\HFz{{\HF\otimes_{\z } H_F^o}}
\def\HFzm{{\HF\otimes_{\z} \HF}}

\def\HCz{{\HC\otimes_{\z} H_C^o}}

\def\HCzm{{\HC\otimes_{\z} {\HC}}}

\newtheorem{theorem}{Theorem}[section]
\newtheorem{corollary}[theorem]{Corollary}
\newtheorem{lemma}[theorem]{Lemma}
\newtheorem{proposition}[theorem]{Proposition}
\newtheorem{conjecture}[theorem]{Conjecture}
\newtheorem*{theorem*}{Theorem}
\newtheorem*{proposition*}{Proposition}
\newtheorem{fact}[theorem]{Fact}

\theoremstyle{remark}

\newtheorem{definition}[theorem]{Definition}

\newtheorem{remark}[theorem]{Remark}

\theoremstyle{definition}

\makeatletter
\def\@tocline#1#2#3#4#5#6#7{\relax
  \ifnum #1>\c@tocdepth % then omit
  \else
    \par \addpenalty\@secpenalty\addvspace{#2}%
    \begingroup \hyphenpenalty\@M
    \@ifempty{#4}{%
      \@tempdima\csname r@tocindent\number#1\endcsname\relax
    }{%
      \@tempdima#4\relax
    }%
    \parindent\z@ \leftskip#3\relax \advance\leftskip\@tempdima\relax
    \rightskip\@pnumwidth plus4em \parfillskip-\@pnumwidth
    #5\leavevmode\hskip-\@tempdima
      \ifcase #1
       \or\or \hskip 1em \or \hskip 2em \else \hskip 3em \fi%
      #6\nobreak\relax
    \dotfill\hbox to\@pnumwidth{\@tocpagenum{#7}}\par
    \nobreak
    \endgroup
  \fi}
\makeatother

\title{Rigid dualizing complexes of affine Hecke algebras}
\author{Sabin Cautis, Rachel Ollivier}

\address{University of British Columbia,  1984 Mathematics Road, Vancouver, BC V6T 1Z2, Canada}
\email{cautis@math.ubc.ca}
\email{ollivier@math.ubc.ca}

\begin{document}

\begin{abstract}
We identify the rigid dualizing complex of the (generic) affine Hecke algebra $H_\q$ attached to a reduced root system and deduce some structural properties as a consequence. For example, we show that the classical Hecke algebra $H_{\q^\pm}$ as well as  $H_\q/\q$ are, under a certain condition on the root system, Frobenius over their centers with Nakayama automorphism given by an explicit involution $\upiota$. 
\end{abstract}

\maketitle

\setcounter{tocdepth}{2}

\tableofcontents

\section{Introduction}

In the introduction we work over a fixed field $\k$ (in the rest of the paper we work over a more general ring $\r$). Attached to a reduced based root system $(X,\Phi,\cX,\cPhi,\Pi)$ one has an affine Coxeter system $(W_\aff, S_\aff)$ and an extended affine Weyl group $W$. The affine Hecke algebra $H_\q$ associated to this sytem is the $\k[\q]$-algebra with basis $\{\tau_w\}_{w \in W}$ and relations
\begin{align*}
	\tau_{w} \tau_{w'} =  \tau_{ww'} \qquad & \text{if $\ell(ww')=\ell(w)+\ell(w')$} \\
	 \tau_{s}^2=(\q-1)\tau_s + \q \qquad & \text{for $s \in  S_\aff$}
\end{align*}
where $\ell$ is the length function on $W$ arising from $(W_\aff,S_\aff)$.

If we invert $\q$ we get the more common affine Hecke algebra $H_{\q^\pm}$ studied in the complex representation theory of $p$-adic reductive groups and in geometric representation theory (\cite{Lu}, \cite{KL} for example). These algebras can be recovered geometrically from categories of constructible sheaves on affine flag varieties or from categories of coherent sheaves on Steinberg varieties. If we set $\q=0$ we get Hecke algebras $H_0$ that appear naturally in the mod-$p$ representation theory of $p$-adic reductive  groups and geometrically as coherent sheaves on affine flag varieties. 

We study a natural graded version $\Hh$ of $H_{\q}$ (\S\ref{subsec:Hecke}) which is defined over $\k[\a,\b]$ with quadratic relations
$$(T_s-\a)(T_s-\b)=0 \qquad \text{ for $s \in S$}.$$
This algebra can be interpreted as the Rees algebra of $H_\q$ with respect to a natural filtration by length. One recovers $H_\q$ by setting $T_s=-\tau_s,\a=-\q,\b=1$.

The algebra $\Hh$ is equipped with an involution $\upiota$ which fixes $\k[\a,\b]$ and satisfies $\upiota(T_s-\a)=-(T_s-\b)$  (see \eqref{f:upiota}). The main result of this paper is the following explicit identification of the rigid dualizing complex $\R_{\Hh}$ of $\Hh$.

\begin{theorem}\label{thm:main}
As $\Hh$-bimodules we have $\R_\Hh \cong (\upiota) \Hh [\rk(X)+2]$ where $(\upiota)$ denotes the left action twisted by $\upiota$. 
\end{theorem}
Rigid dualizing complexes \cite{Ber} (along with earlier work on balanced dualizing complexes \cite{Ye1}) are an attempt to extend Grothendieck duality to non-commutative algebras. In particular, if $A$ is a commutative algebra of finite type over $\k$, its rigid dualizing complex corresponds to $\pi^!(\k)$ where $\pi: \Spec A \to \Spec \k$. We review dualizing complexes in Section \ref{sec:rigiddualizing}.
 
Theorem \ref{thm:main} above follows from Corollary \ref{coro:main} where the $+2$ in the shift is a reflection of working over $\k[\a,\b]$. By base change (Corollary \ref{cor:basechange}), Theorem \ref{thm:main} also identifies the rigid dualizing complexes of related algebras such as $H_{\q^\pm}$ or $H_0$. For example, $\R_{H_0} \cong  (\upiota) H_0 [\rk(X)]$. 

One of the main implications of Theorem \ref{thm:main} is that, under a certain condition on the root lattice $Q = \Z[\Phi]$,  the algebras $H_0$ and $H_{\q^\pm}$ are (free) Frobenius algebras over their centers (Corollary \ref{cor:frob}).

\begin{corollary}\label{cor:main}
If $X/Q$ is a free abelian group then $H_{\q^\pm}$ and $H_0$ are Frobenius algebras over their centers with Nakayama automorphism $\upiota$. 
\end{corollary}

The centers  of $H_{\q^\pm}$  and $H_0$ are known to be isomorphic to  $\k[\cX]^{W_0}[\q^\pm]$ (\cite{Lu}) and  $\k[\cX^+]$  (\cite{Ollcompa}) respectively, where $W_0$ is the finite Weyl group and $\cX^+$ the semigroup of dominant coweights. This suggests that the center of $\Hh$ should be isomorphic to $\k[\a,\b][\cX]^{W_0}$ (Conjecture \ref{conj:center}). Assuming this conjecture, Theorem \ref{thm:main} likewise implies that $\Hh$ is Frobenius over its center with Nakayama automorphism $\upiota$ (again if $X/Q$ is free). 

The Frobenius structure from Corollary \ref{cor:main} is difficult to see directly. When the root system is   associated to $\rm {SL}_2$, the trace morphism from $H_0$ to its center was worked out explicitly in \cite[Prop. 2.13]{OS2}. Generally we do not have an explicit description of the trace maps from $H_{\q^\pm}$ or $H_0$ to their centers. This is because the general structure of $\Hh$ over its center is difficult to study. For example, it is easy to see that in general $\Hh$ cannot be a matrix algebra over its center because one can find one dimensional characters such as $T_s \mapsto a$, $\a \mapsto a$ for any $a \in \k$.  

Even deciding if $\Hh$ is projective over its center is non-trivial. Classical results (\cite{Lu}) tell us that $H_{\q^\pm}$ contains a commutative subalgebra $A_{\q^\pm}$ over which it is finite and projective. If $X/Q$ is free then $A_{\q^\pm}$ can be shown to be projective over the center which explains why $H_{\q^\pm}$ is projective over its center. Likewise, $H_0$ also contains a natural commutative subalgebra $A_0$ (\cite{Vigann}) but $H_0$ is no longer free (or even flat) over $A_0$ (\cite{GL3}). Nevertheless, assuming $X/Q$ is free, $H_0$ remains projective over its center (Proposition \ref{prop:projective}). The argument for this is indirect, using Theorem \ref{thm:main} together with the non-commutative version of Hironaka's criterion (or miracle flatness) from Proposition \ref{prop:miracle}. 

The representation theory of $\Hh$ is sensitive to the parameters $\a,\b$. The fact that $(\upiota) \Hh[\rk(X)+2]$ is a rigid dualizing complex captures many features of $\Hh$ in a uniform way. For example, by Proposition \ref{prop:fin-module}, if $M$ is a finite dimensional $\Hh$-module then 
$$\Ext_{\Hh}^i(M,\Hh) \cong 
\begin{cases}
(\upiota) M^\vee & \text{ if } i = \rk(X)+2 \\
0 & \text{ otherwise }
\end{cases}$$
Similar results hold by base change for $H_{\q^\pm}$ and $H_0$. This result was obtained previously in \cite[Cor. 6.12, Cor. 6.17]{OS1} by more detailed analysis of the  Iwahori-Hecke $k$-algebra of a connected split semisimple group $G$ over a $p$-adic field. The argument above explains how it is a formal consequence of Theorem \ref{thm:main}. 

The key tool used to prove Theorem \ref{thm:main} is a certain resolution of $\Hh$ constructed in Section \ref{sec:reso} as follows. Consider the Coxeter complex $\Aa$ associated to $(\Waff, \Saff)$. It is a polysimplicial complex of dimension $d=\rk(Q)$. We define a right $\Hh$-equivariant coefficient system $\CH$ on $\Aa$ which  is locally constant in the sense of Remark \ref{rema:locc}. It yields an augmented complex  
\begin{equation}\label{f:complexCint}
0 \longrightarrow C_c^{or} (\mathscr{A}_{(d)}, \CH) \longrightarrow \ldots \longrightarrow C_c^{or} (\mathscr{A}_{(0)}, \CH)   \longrightarrow \Hh \longrightarrow 0 \ 
\end{equation}  
of oriented chains which is exact because of the contractibility of certain subcomplexes of $\Aa$. In Proposition \ref{prop:isoBimo} we show that 
$$C_c^{or} (\mathscr{A}_{(i)}, \CH) \cong \bigoplus_{F \in \mathscr F_i} \Hh (j_F) \otimes _\HF \Hh$$
where $\mathscr F_i$ is a set of facets of dimension $i$ in the standard chamber, $\HF$ is the (finite) Hecke algebra attached to $F$ and $(j_F)$ is the twist of the action of $\HF$ by a certain orientation character. Subsequently, \eqref{f:complexCint} becomes an exact complex of $\Hh$-bimodules (cf. Corollary \ref{cor:main1})
\begin{equation}\label{f:complexHint}
0 \longrightarrow \bigoplus_{F\in \mathscr F_d}\Hh(j_F)\otimes_{\HF} \Hh \longrightarrow \ldots \longrightarrow\bigoplus_{F\in \mathscr F_0}\Hh(j_F)\otimes_{\HF} \Hh\longrightarrow  \Hh\longrightarrow 0.
\end{equation}

This resolution is inspired by \cite{OS1} which proves \eqref{f:complexHint} for the Iwahori Hecke $k$-algebra of a connected split reductive group $G$ over a $p$-adic field $\mathfrak F$, namely for the specialization of  $H_{\q}$ to  $\q\mapsto p^f$ where $p^f$ is the size of the residue field of $\mathfrak F$. The construction in \cite{OS1} is different in that one starts from a resolution of the smooth representation $k[G/I]$ of $G$ given in \cite{SS}, where $I$ is an Iwahori subgroup of $G$. Then passing to $I$-invariants gives a resolution of $H_{p^f} \cong k[I\backslash G/I]$ as an $H_{p^f}$-bimodule. 
 
\subsection{Acknowledgements}
We would like to thank Amnon Yekutieli and Eric Vasserot for their generous answers that helped clarify our thoughts. The authors are grateful for the support of the Fondation des Sciences Math\'ematiques de Paris, Universit\'e d'Orsay and the Ecole Normale Sup\'erieure-PSL for support and a serene and stimulating work environment. 

\section{Affine Hecke algebras} \label{sec:not}

\subsection{Based root systems} \label{sec:rootsystem}

Consider a (reduced) based root system $(X, \cX, \Phi, \check\Phi, \Pi)$  (cf. \cite[1.1]{Lu}). Then $X$ and $\cX$ are free abelian groups of finite rank equipped with a perfect pairing $\la -,- \ra : X \times \cX \to \Z$. The finite sets $\Phi \subset X$ and $\cPhi \subset \cX$ are the sets of roots and coroots. There is a bijection $\alpha \leftrightarrow \check\alpha$ such that $\la \alpha ,\check\alpha \:\ra = 2$. For  every $\alpha\in \Phi$, the reflections 
 $$s_\alpha: X\rightarrow X ,\: x\mapsto x- \langle x, \check\alpha\rangle\alpha\quad \text{ resp. \:\:\:$ s_{\check\alpha}: \cX\rightarrow \cX , \: \check x\mapsto \check x- \langle  \alpha , \check x \rangle \check\alpha$}$$ 
preserve $\Phi$ and $\Phi^\vee$ respectively.  The base $\Pi \subset \Phi$ consists of simple roots and defines the sets $\Phi^+$ and $\Phi^-$ of positive and negative roots.

Denote by $Q = \Z[\Phi]$ the root lattice. We let $Q^\perp:=\{\check  x \in \cX: \la \alpha, \check x \ra=0, \forall \alpha \in \Phi\}$ and 
$\cX^+ = \{\check x \in \cX: \la \alpha, \check x \ra \ge 0, \forall \alpha \in \Phi \}$  the set of dominant  coweights.

Define the set of affine roots by  ${\Phi_\aff}=\Phi\times \mathbb Z={\Phi_\aff^+}\coprod {\Phi_\aff^-}$ where
$${\Phi_\aff^+}:=\{(\alpha , r),\: \alpha \in\Phi, \,r>0\}\cup\{(\alpha ,0),\, \alpha \in\Phi^+\}, \quad{\Phi_\aff^-}:=\{(\alpha , r),\: \alpha \in\Phi, \,r<0\}\cup\{(\alpha ,0),\, \alpha \in\Phi^-\}.$$
There is  a partial order on $\Phi$ given by $\alpha \preceq \beta$ if and only if $\beta -\alpha $ is a linear combination with (integral) nonnegative coefficients of elements in $\Pi$. Denote by $\Pi _m$ the set of roots that are minimal elements  for $\preceq$. The set of simple affine roots is  $\Pi _\aff:=\{(\alpha , 0),\: \alpha \in\Pi \}\cup\{(\alpha ,1),\, \alpha \in\Pi _m\}$. 

We denote by $W_0$ the finite Weyl group, namely the subgroup of ${\rm GL}(X)$ generated by $\{s_\alpha\}_{\alpha\in \Phi}$. Let $S_0:=\{s_\alpha\}_{\alpha\in \Pi} $. Then $(W_0, S_0)$ is a (finite) Coxeter system. The extended affine Weyl group is $W = W_0 \ltimes \cX$. An element  $w_0 \check x\in W_0 \ltimes \cX$ acts on $\Phi_\aff$ by $w_0\check x: (\alpha , r)\mapsto (w_0\alpha, r- \la  \alpha, \check x \ra)$.

The length $\ell$ on the Coxeter system $({W_0}, S_0)$ extends to $W $ in such a way that,  the length of $w\in W $ is the number of affine roots   $A \in {\Phi_\aff^+}$ such that $w(A) \in \Phi_\aff^-$. For any $A \in \Pi _\aff$ and $w \in W$ it satisfies
\begin{equation}\label{add}
\ell(w s_A)= 
   \begin{cases}
       \ell(w)+1 & \textrm{ if }w (A)\in {\Phi_\aff^+},\\  \ell(w)-1 & \textrm{ if }w (A)\in {\Phi_\aff^-}.
    \end{cases}
\end{equation}
where $s_A$ is the  affine reflection associated to $A$. 

Let $  \Saff := \{s_A\}_{A \in \Pi_\aff}$ and define the affine Weyl group as $\Waff := \la s_A  \ra_{A \in \Phi_\aff}  \subset W$. The pair $(\Waff, \Saff)$ is a Coxeter system with length function $\ell$ (\cite[V.3.2 Thm.\ 1(i)]{Bki-LA}). If $\Omega \subset W$ is the abelian subgroup consisting of length zero elements then $W \cong \Omega \ltimes \Waff$  (\cite[1.5]{Lu}). 

The action of $\Omega$ on $W$ by conjugation preserves $S_\aff$. Consider the character 
\begin{equation}
 W\rightarrow W/\Waff\cong \Omega \rightarrow \{\pm 1\}\label{orichar}
 \end{equation}
 where the second map is the signature of $\Omega$ acting on $S_\aff$. In \S\ref{subsec:finiteH} we will denote this character by $\epsilon_C$ and it will be used in Remark \ref{rema:JC} to define the involution \eqref{f:upiota}. 
 
\subsection{Generic Hecke algebras: definitions and basic properties\label{subsec:Hecke}}

Fix a Noetherian ring $\r$.  The  extended affine Hecke algebra $H_\q$ is the $\r[\q]$-algebra which is free as an $\r[\q]$-module with basis $\{ \tau_w\}_{w \in W}$ and subject to relations 
\begin{align*} 
\tau_v  \tau_w =  \tau_{vw} & \quad \text{ if } \ell(vw)=\ell(v)+\ell(w)  \\
(\tau_s - \q)( \tau_s + 1)=0 & \quad \text{ for  } s \in \Saff.
\end{align*}
We denote $H_{\q^\pm} := H_\q \otimes_{\r[\q]} \r[\q^{\pm 1}]$ and $H_0 := H_\q/\q$. There is a natural filtration on $H_\q$ such that $F_iH_\q$ is the free $\r$-module with basis $\{\q^j  \tau_w: j+\ell(w) \le i\}$. The associated Rees algebra with parameter $x$ is the graded algebra $\Rees(H_\q) := \bigoplus_{i \ge 0} x^i F_iH_\q$ equipped with the natural multiplication.

On the other hand, consider the $\r[\a,\b]$-algebra $\Hh$ which is a free $\r[\a,\b]$-module with basis $\{T_w, \: w \in W\}$ and subject to relations
\begin{align}
\label{braid} T_v T_w = T_{vw} & \quad \text{ if } \ell(vw)=\ell(v)+\ell(w) \quad \emph{(braid relations)} \\
\label{quad} (T_{s}-\a)(T_s-\b)=0 & \quad \text{ for } s \in \Saff \quad \emph{(quadratic relations)}
\end{align}
It is generated as an $\r[\a,\b]$-algebra by $\{T_s\}_{s \in \Saff}$ and $\{\tau_\omega\}_{\omega \in \Omega}$. One can check that $\Hh \cong \Rees(H_\q)$ by taking $\a \mapsto - \q x$, $\b \mapsto x$ and $T_w \mapsto (-x)^{\ell(w)} \tau_w$ for $w \in W$.

\begin{remark}  
By \cite[IV Exercices \S2, 23]{Bki-LA}, there is a unique $\r[\a,\b]$-algebra $H_\aff$ with basis $\{T_w\}_{w \in \Waff}$ satisfying \eqref{braid} and \eqref{quad}. One can then check that $\Hh \cong H_\aff \otimes_{\r[\a,\b]} \r[\a,\b][\Omega]$ with the product on the right given by $(T_v \otimes T_\omega) \cdot (T_{v'}\otimes T_{\omega'}) = (T_v T_{\omega v'\omega^{-1}}) \otimes (T_\omega T_{\omega'})$.
\end{remark}

\begin{remark}\label{rem:H00}
Notice that $H_{\a,1} \cong H_\q$ and $H_{0,1} \cong H_0$ while $H_{0,0}$ recovers the affine nil-Coxeter algebra. 
\end{remark}

\begin{proposition}
There exists an $R[\a,\b]$-linear algebra involution $\iota$ of $\Hh$ determined by
\begin{equation}\label{f:invo}
 \iota(T_\omega)= T_\omega \text{ if $\ell(\omega)=0$}\quad\text{ and }\quad \iota(T_s-\a)=-(T_s-\b) \text{ for $s\in \Saff$ } .
 \end{equation}
\end{proposition}
\begin{proof}
Recall that $H_{\q^\pm}$ is equipped with a $\r[\q]$-linear algebra involution $\tau _w \longmapsto (-\q)^{\ell(w)}  \tau _{w^{-1}}^{-1}$ which restricts to an involution $\iota_1$ of $H_\q$ (\cite[Cor. 2]{Vigann}). This involution preserves the filtration $(F_iH_\q)_{i\geq 0}$ and induces an involution $\iota$  on $\Rees(H_\q) \cong \Hh$. Since $\iota_1( \tau_\omega) =  \tau_\omega$ if $\ell(\omega)=0$ and $\iota_1(\tau _s)= \q-1-\tau _s$ for $s \in \Saff$ it follows that $ \iota(T_\omega)= T_\omega$ if $\ell(\omega)=0$ and $\iota(T_s)=\a+\b-T_s $ for $s \in \Saff$.
\end{proof}

\begin{proposition}\label{prop:finite}
The algebra $H_{0,0}$ is finite over its center which itself is a finitely generated $\r$-algebra. 
\end{proposition}
\begin{proof}
The relations in $H_{0,0}$ are
\begin{equation*}T_{v} T_w= \begin{cases}T_{vw}& \text{ if  $\ell(vw)=\ell(v)+\ell(w)$}\cr0&\text{ otherwise} \end{cases}\end{equation*}
for any $v,w \in W$. It contains the commutative subalgebra $A_{0,0}$ with $\r$-basis $\{T_{\check x}\}_{{\check x} \in \check X}$. Following  the proof of \cite[Prop. 8.5]{OS1}  we find that $H_{0,0}$ is finite over $A_{0,0}$ which is itself a finitely generated $\r$-algebra. Consider the subalgebra $ A_{0,0}^{W_0}$ where $W_0$ acts via its natural action on $\cX$.  Then % (\cite[V 1.9 Thm. 2]{Bki-AC})  
$A_{0,0}^{W_0}$ is a finitely generated $\r$-algebra and $A_{0,0}$ is finite over it. Therefore, $H_{0,0}$ is finite over  $A_{0,0}^{W_0}$. 

We now check that $A_{0,0}^{W_0}$ is central in $H_{0,0}$. For $\omega \in \Omega$ and ${\check x} \in \check X$ we have $T_\omega T_{\check x} = T_{\omega {\check x} \omega^{-1}} T_\omega$ where $\omega {\check x} \omega^{-1}$ is a $W_0$-conjugate of ${\check x}$. Thus $T_\omega$ commutes with $A_{0,0}^{W_0}$. On the other hand, for $s \in \Saff$ we have $\ell(s{\check x})= \ell({\check x})+1$ (resp.  $\ell({\check x}s)= \ell({\check x})+1$) if and only is $\la \alpha_s, {\check x} \ra \geq 0$ (resp. $\la \alpha_s, {\check x} \ra \leq 0$). In this case $T_s T_{\check x}= T_{s{\check x}}$ (resp. $ T_{\check x} T_s = T_{{\check x}s}$). Otherwise $T_{s} T_{\check x}=0$ (resp. $T_{{\check x}} T_s=0$). This shows that $T_s$ commutes with the subspace $A_{0,0}^{\la s \ra}$ of $s$-invariants in $A_{0,0}$ and therefore it commutes with $A_{0,0}^{W_0}$.

Thus $A_{0,0}^{W_0}$ is contained in the center of $H_{0,0}$. Since $H_{0,0}$ is finite over $A_{0,0}^{W_0}$ it is also finite over its center. Moreover,  $A_{0,0}^{W_0}$ is Noetherian (since $\r$ is Noetherian) which   implies that the center is also finite over $A_{0,0}^{W_0}$. This means that the center of $H_{0,0}$ is a finitely generated $\r$-algebra.
\end{proof}

\begin{corollary} \label{coro:noeth}
The algebra $\Hh$ is Noetherian.
\end{corollary}
\begin{proof}
Recall that $\Hh = \Rees(H_\q)$ is a graded algebra with $\a,\b$ homogeneous of degree one. By Proposition \ref{prop:finite} we know $H_{0,0}$ is Noetherian. It follows that $\Hh$ is also Noetherian by applying \cite[Lemma  8.2]{ATVdB}.
\end{proof}

We denote $\z:=\r[\a,\b][{Q^\perp}] \subset \Hh$. This is a central subalgebra of $\Hh$. Since $Q^\perp \subset W$ is a free group of finite rank $\z$, is a Laurent series ring over $\r[\a,\b]$. 

\subsection{Coxeter complexes}\label{313}

Consider the affine Coxeter complex $\Aa$ associated to $(\Waff, \Saff)$ (cf. \cite[Ch. V. \S3]{Bki-LA} and \cite[I.3.1]{BT}). It is a polysimplicial complex of dimension $d := \rk(Q)$ (\cite[I.3.4]{BT}). The group $W$ acts on $\Aa$ and we  call $C$ the chamber  of $\Aa$ whose stabilizer in $W$ is $\Omega$ (\cite[VI, \S2.3]{Bki-LA}). 

\begin{remark}\label{rema:4.4}
The group $\Waff$ acts simply transitively  on the chambers  and given a facet ${F}$ there is a unique facet contained  in $\overline C$ which is $\Waff$-conjugate to $F$ (\cite[V.3.2 Thm 1]{Bki-LA}, \cite[I.3.5]{BT}).
\end{remark}

Let $i>0$ and $F$ a facet of dimension $i$. The $(i+1)!$ arrangements of the $i+1$ vertices of $F$ decompose into two classes under the action of the even permutations. These two classes  are called the orientations of  $F$. 
%See for example \cite[Ch. 2 \S7]{WF}. 
For a formal  definition of the oriented facets of dimension $\geq 0$, we refer to \cite[II.1]{SS}. We will denote an oriented facet by $(F, c)$ as in \loccit  with the convention that the $0$-dimensional faces always carry the trivial orientation. A facet $F$ of dimension $\geq 1$  has two possible orientations, $(F,c)$ and $(F,-c)$.   The  orientation $(F,c)$ induces an orientation $(F, c)|_{F'}$ on each of the facets  $F'$ of $F$ of dimension $i-1$. 
It satisfies $(F, -c)|_{F'}=-(F, c)|_{F'}$.

\medskip

To  a facet ${{F}}$ contained in $\overline C$ we associate the subset $S_{{F}} \subset \Saff$ fixing ${{F}}$  (pointwise) and the corresponding subset $\Pi_{{F}} \subset \Pi$. Let  $W^0_{{F}} \subset \Waff$ be the finite subgroup generated by $S_{{F}}$ (\cite[V.3.6 Prop.\ 4]{Bki-LA}). The pair $(W ^0_{{F}}, S_{{F}})$ is a Coxeter system with length function $\ell|_{W^0_F}$ (\cite[IV.1.8 Cor.\ 4]{Bki-LA}). We let $\Phi_F := \{ A \in \Phi_\aff: s_A \text{ fixes } F\}$ and $\Phi _{{F}}^+ := \Phi _{{F}} \cap \Phi ^+_\aff$. 

\begin{proposition}Let ${{F}}$ be  a facet contained in  $\overline C$.\label{representants}
\begin{itemize}
\item[i.] The set $\D_{{F}} := \{ d \in W: d(\Phi_{{F}}^+) \subset {\Phi}_\aff^+ \}$ is a system of representatives of the left cosets $W/ W^0_{{F}}$. It satisfies
\begin{equation}\label{additive0}
\ell( dw)=\ell(d)+\ell(w)
\end{equation}
for any $w\in W^0_{{F}}$ and $d \in \D _{{F}}$. In particular, $d$ is the unique element with minimal length in $d\WFf$.
\item[ii.] If $s \in S_\aff$ and $d \in \D _{{F}}$ then we are in one of the following situations:
\begin{itemize}
\item[$\textrm{\bf{A}.}$] $\ell(s d)=\ell(d)-1$ in which case $sd\in \D _{{F}}$.
\item[$\textrm{\bf{B}.i.}$] $\ell(sd)=\ell(d)+1$ and $sd\in \D _{{F}}$.
\item[$\textrm{\bf{B}.ii.}$] $\ell(sd)=\ell(d)+1$ and $sd\in d\WFf$.
\end{itemize}
\end{itemize}
\end{proposition}
\begin{proof}
This result follows as in \cite[Prop. 4.6]{OS1}. The main point to note is that the argument for (ii) goes through with $s \in S_\aff$ even if only the case $s \in S_F$ is considered in \cite{OS1}. 
\end{proof}

Denote by $\Omega_{{F}} \subset \Omega$ the subgroup stabilizing $F$. We have $\Omega_{{F}}=\{ \omega\in \Omega, \: \omega S_{{F}} \omega^{-1}= S_{{F}}\}$  and  $ \WFf$ is normalized by $\Omega_{{F}}$ (\cite[\S 4.5]{OS1}). Denote by $\WF \subset W$ the subgroup generated by  $ \WFf$ and  $\Omega_{{F}}$ (it is a semi-direct product of these two subgroups). Note that $\WC=\Omega$ and $\WF\cap \Waff=\WFf$. 
By \cite[Lemma 4.9]{OS1},  we have $\WF = \{w \in W: wF = F\}$.

\begin{proposition} \phantomsection \label{prop:C(F)}
\begin{itemize}\item [i.]
Let $F$ be a a facet of $\Aa$. There a  unique 
chamber $C(F)$   at  closest distance to $C$   (in terms of gallery distance) which contains $F$ in its closure and  a unique
element $d_F$  in $\Waff$ such that $C(F)=d_F C$. We have $d_F\in \D_{F_0}$ where $F_0:= d_F^{-1} F$.
\item[ii.]If  $F$ and $F'$ are two facets such that $F'\subseteq \overline F$ we have $\ell(d_{F'})+\ell( d_{F'}^{-1} d_F)=\ell(d_{F})$ .\end{itemize}
\end{proposition}

\begin{proof} Let $F$ be a facet of $\Aa$.  By 
\cite[Prop. 4.13-i]{OS1},  there is a  unique 
chamber $C(F)$   at  closest distance to $C$   (in terms of gallery distance) which contains $F$ in its closure. 
Since $\Waff$ acts simply transitively  on the chambers, there is a unique $d_F\in \Waff$ such that 
$C(F)=d_FC$. 

 Let $F_0:= d_F^{-1} F$ and  $d\in \D_{F_0}$ such that $d_F= du\in d W_{F_0}^0$. Then $dC$ contains $F$ in its closure and $\ell(d)=\ell(d_F)-\ell(u)$ of $C$. Therefore $u=1$ and $d_F=d\in \D_{F_0}$. 
Now let   $F'$  a facet such that $F'\subseteq \overline F$.
Both facets   $F'_0:=d_{F'}^{-1} F'$  and $d_F^{-1} d_{F'} F'_0$ are contained in $\overline C$.
By Remark  \ref{rema:4.4} we have
$d_F^{-1} d_{F'} F'_0= F'_0$ and $d_{F'}^{-1}d_F \in W^0_{F'_0}$. Point ii then follows from \eqref{additive0}.

\end{proof}
\begin{definition}\label{defi:types}
\noindent Given  $s\in \Saff$, a facet $F$  of $\Aa$ and $F_0:= d_F^{-1}F\subseteq \overline C$ we say that $F$ is 
\begin{itemize}
\item of type $\mathbf A$  for $s$ if $\ell(sd_F)=\ell(d_F)-1$ in which case  we recall that $sd_F\in \D_{F_0}$,
\item  of type $\mathbf {B.i}$  for $s$ if $\ell(s d_F)= \ell(d_F)+1$ and $sd_F\in \D_{F_0}$,
\item of type $\mathbf{B.ii}$  for $s$ if $\ell(s d_F)= \ell(d_F)+1$ and $sd_F\not\in \D_{F_0}$. Equivalently,  $F$ is of type $\Bii$  for $s$ when $sF=F$.
\end{itemize}
\end{definition}

\begin{lemma} Given  two facets $F',F$   such that $F'\subseteq \overline F$ and $s\in \Saff$,
\begin{itemize}
\item if $F$ is of type $\A$,   then $F'$ is of type $\A$  or $\Bii$ for $s$.
\item if $F$ is of type $\Bi$,  then $F'$ is of type $\Bi$ or $\Bii$ for $s$.
\item if $F$ is of type $\Bii$, then $F'$ is of type $\Bii$ for $s$.
\end{itemize}
\end{lemma}
\begin{proof} With the notation of the proof of Proposition \ref{prop:C(F)}, we have  $d_F=d_F' u$ where $u\in W^0_{F_0'}$ and $\ell(d_F)= \ell(d_{F'})+\ell(u)$.\\
If $F'$ is of type $\Bi$, then $\ell(s d_F)=\ell( s d_{F'} u)=\ell( s d_{F'})+\ell(u)=\ell( d_{F'})+1+\ell(u)=\ell( d_F)+1$ so $F$ is not of type $\A$. \\
If $F'$ is of type $\A$, then $\ell(s d_{F})=\ell( s d_{F'} u)=\ell( s d_{F'})+\ell( u)=\ell( d_{F'})-1+\ell(u)=\ell( d_F)-1$ so $F$  is not of type $\Bi$.
Suppose that $F$ is of type $\Bii$,  then $sF=F$ so $sF'=F'$ and $F$ is also of type $\Bii$.
\end{proof}

For $F$ a facet in $\overline C$, the set $\D_F$ is stable by right multiplication by $\Omega_F$. By  \cite[Lemma 4.12-i]{OS1},  a chosen system of representatives $\D_F^\dagger$ of the orbits of $\D_F$ under the right action of $\Omega_F$ is a system of representatives of the left cosets $W/ \WF$.

\subsection{Finite Hecke algebras\label{subsec:finiteH} and orientation characters}

For a facet $F$ in  $\overline C$ consider the free $R[\a,\b]$-submodule $\HF \subset \Hh$ with basis $\{T_w\}_{w \in \WF}$.  It is a subring of $\Hh$.
 \begin{proposition} \label{prop:frees} 
 $\Hh$ is free as an $\HF$-module on the right (resp. left) with basis $\{T_d\}_{d\in\D_F^\dagger}$ (resp. $\{T_{d^{-1}}\}_{d\in\D_F^\dagger}$).
 \end{proposition}
 \begin{proof}
 This follows from Proposition \ref{representants} and \eqref{braid} (cf. \cite[Prop. 4.21]{OS1}).
\end{proof}

\begin{proposition}\phantomsection\label{prop:freesF}
We have 
\begin{enumerate}
 \item $\HF$ is finite and  free over $\z$.  
 \item $\HF$   is  a Noetherian ring.
 \item  As a right (resp. left) $\HF$-module, ${\Hh}$ is free of basis $\{T_d, \: d\in \D_{F}^\dagger\}$ (resp.  $\{T_{d^{-1}}, \: d\in \D_{F}^\dagger\}$).
 \end{enumerate}
\end{proposition}
\begin{proof}

Recall  that $\WF$  is the semidirect product of the finite group $\WFf$ and of the subgroup $\Omega_{F}$ (containing ${Q^\perp}$) of $\Omega$. Hence  $\WF/{Q^\perp}$ is finite and $\HF$ is free over  $\z$ with basis  the set $\{T_w\}$ where $w$ ranges over a system of representatives $[\WF/{Q^\perp}]$. Point 2 follows immediately. Point 3 comes from the definition of $\D_{F}^\dagger$ (see the very end of \S\ref{313} and \eqref{additive0}) and from the braid relations \eqref{braid}. 
\end{proof}

For $w \in W_{F}$, set  $\epsilon_{F}(w) = +1$ (resp.\ $-1$) if $w$ preserves (resp. reverses) a given orientation of $F$. This defines a character $W_{F}\rightarrow\{\pm 1\}$  which is trivial on ${Q^\perp}$ (cf. \cite[3.1 and 3.3.1]{OS1}).
 \begin{lemma}
 The $\r[\a,\b] $-linear map 
\begin{equation}\label{f:defiep}
 j _{F}: \HF\longrightarrow \HF\, ,\quad\quad T_w\longmapsto \epsilon_{F}(w) T_w 
\end{equation} 
is an involutive automorphism of $\HF$ which acts trivially on $\z$.
\end{lemma}

\begin{proof} An element  $w\in W_{{F}}$ fixes ${F}$ pointwise so 
$\epsilon_{F}$ factors through
$\WF/ W_{F}^0\cong \Omega_{F}\rightarrow \{\pm 1\}$; namely, for $w= w_0\omega\in W_{F}^0\rtimes \Omega_{F}=\WF$, we have
 $j _{F}(T_w)=\epsilon_{F}(T_\omega)T_w $. We want to show that $j _{F}(T_v T_w)=j _{F}(T_v) j _{F} (T_w)
$ for all $v,w\in \WF$. By induction it is enough to do it when $v\in \Omega_F$ and $v\in S_F$ which is immediate using \eqref{braid} an \eqref{quad}.
\end{proof}

Given a left (resp. right) $\HF$-module $M$ we will denote by $(j_F)M$ (resp. $M(j_F)$) the left (resp. right) $\HF$-module
$M$ where the action of $\HF$ is twisted by $j_F$. 

\begin{remark}  \label{rema:JC}
Recall that $W_C=\Omega$. The character $\epsilon_C$ of $\Omega$ extends to a  character of $W$ as in \eqref{orichar}. It is easy to check that the  involution $j_C$ of $\HC$ extends to an involution  of $\Hh$ which we still denote by $j_C$ and which is given by
\begin{equation}\label{jC}j_C: \Hh\longrightarrow \Hh, \quad T_{\omega w}\mapsto \epsilon_C(\omega w) T_{\omega w}=\epsilon_C(\omega)  T_{\omega w} \text{  for  $\omega \in \Omega$ and $w\in \Waff$. }\end{equation}
Define   \begin{equation}\label{f:upiota}\upiota:= \iota\circ j_C\ \end{equation} 
where  $\iota$   was given in \eqref{f:invo}. 
It is an involution of $\Hh$ which acts trivially on $\z$. 
\end{remark}

\subsection{Coefficient systems}

Following  \cite[II.2]{SS} we define a (contravariant) coefficient system $\cV$  of right $\Hh$-modules on $\Aa$ as the following data: a right $\Hh$-module $\EuScript M_F$  for each facet $F$ of  $\Aa$ and  a right $\Hh$-equivariant transition map
$r^F_{F'}: \V_F\rightarrow \V_{F'}$ for each pair of facets $F$ and $F'$ such that $F'\subseteq \overline F$    satisfying
 \begin{center}$r_F^F= \id_{\V_F}$, and $r^F_{F''}=r^{F'}_{F''}\circ r^F_{F'}$ whenever  $F'\subseteq \overline F$ and 
 $F''\subseteq \overline F'$.\end{center}
For $i\geq 0$ we denote by $\Aa_i$ (resp. $\Aa_{(i)}$) the set of facets  (resp. oriented faces) of dimension  $i$ of $\Aa$.
To a  coefficient system $\cV$ as above, we attach the oriented chain complex  of right $\Hh$-modules
\begin{equation}\label{f:apart-complexV}
    0 \longrightarrow C_c^{or} (\mathscr{A}_{(d)}, \cV) \xrightarrow{\;\partial\;} \ldots \xrightarrow{\;\partial\;} C_c^{or} (\mathscr{A}_{(0)}, \cV)  \longrightarrow 0
\end{equation}
where, for $i\in\{0, \dots, d\}$, the space of  oriented $i$-chains
$C_c^{or} (\mathscr{A}_{(i)}, \cV)$ consists of  all finitely supported maps\\
$\gamma:  \mathscr{A}_{(i)}\rightarrow  \coprod _{F\in \Aa_i}  V_F$ such that
for any oriented facet $(F,c)$ of dimension $i$ we have
\begin{itemize}
\item $\gamma((F, c))\in \V_F$  and if $i\geq 1$,
\item $\gamma((F, -c))= -\gamma((F, c))$,

\end{itemize} and  when $i\geq 1$,  the differential 
$\partial  \::\:C_c^{or} (\mathscr{A}_{(i)}, \cV)\longrightarrow C_c^{or} (\mathscr{A}_{(i-1)}, \cV)$
is given by
\begin{align*}
\gamma\longmapsto [(F',c')\mapsto \sum_{(F,c)\in \mathscr{A}_{(i)}, F'\subseteq \overline F} r^F_{F'}(\gamma(F,c_{(F',c')}))]
\end{align*} where $(F,c_{(F',c')})$ is  the orientation on $F$ which induces the orientation   $(F',c')$  on $F'$.

\subsection{\label{sec:reso}A resolution for generic Hecke algebras}
 In Proposition \ref{prop:C(F)}, we introduced, for a facet $F$ of $\Aa$, the chamber 
$C(F)$  which   is closest to $C$ in terms of gallery distance and such that $F\subset \overline{C(F)}$ and the element $d_F\in \Waff$ such that $C(F)= d_F C$. For facets $F$ and $F'$ such that $F'\subseteq \overline F$, we consider the right $\Hh$-equivariant maps
 \begin{align}\label{f:defir}r^{F}_{F'}: \Hh&\longrightarrow \Hh\cr h& \longmapsto {T_{d_{F'}^{-1} d_F} }h\end{align}  

\begin{remark}
 If $C(F')= C(F)$ then  $r^F_{F'}$ is the identity map.
\label{rema:locc}
\end{remark}

The following technical lemma will be used in the proof of Proposition \ref{prop:isoBimo}. We refer to Definition \ref{defi:types}. For $w\in W$, we see  $T_w$ below as a the  right $\Hh$-equivariant map of left multiplication by $T_w$ on $\Hh$.

\begin{lemma}\label{lemma:tech}
 Let $F$ a facet of dimension $>0$ and $F'$ a facet in $\overline F$  of codimension $1$.
Let $s \in S_\aff$.

\begin{itemize}
\item[$(a)$] If $F$  and $F'$ are of same type for $s$, then $r^{sF}_{sF'}= r^{F}_ {F'}$.
\item[$(b)$] If $F$ is of type $\A$ and $F'$ of type $\Bii$ for $s$, then
$(\a+\b) r^{F}_{F'} -\a\b \, r^{sF}_ {F'}= T_{d_{F'}^{-1} s d_{F'}}  \circ r^F_{F'}$.
\item[$(c)$] If $F$ is of type $\Bi$ and $F'$ of type $\Bii$ for $s$, then
$ r^{sF}_{F'}= T_{d_{F'}^{-1} s d_{F'}}  \circ r^F_{F'}$.
\item[$(d)$] If $F$ is of type $\Bii$ then 
$ r^{F}_{F'}\circ T_{d_{F}^{-1} s d_F} = T_{d_{F'}^{-1} s d_{F'}} \circ  r^F_{F'}$.
\end{itemize}

\end{lemma}
\begin{proof} 
We make a preliminary remark: for $F$ a facet in $\Aa$, $F_0:= d_F^{-1} F$ and
$s\in \Saff$,  we have $sd_F\in \D_{F_0}$ if and only if $sF\neq F$. More precisely, if  $s d_F\in \D_{F_0}$ then  $d_{sF}= s d_F$ and  in particular $s F\neq F$.
Otherwise,  by Proposition \ref{representants}-ii, we have $sd_F= d_F u$ with $u\in W^0_{F_0}$  and in particular $sF=F$.

\noindent We turn to the proof of the Lemma.\\
For  $(a)$, notice that if none of $F$ and $F'$ is of type $\Bii$ for $s$ then  $d_{sF}= s d_F$, $d_{sF'}= s d_{F'}$. If they are both of type $\Bii$ then $F=sF$ and $F'=sF'$.  \\For the other properties we introduce $A\in \Pi_\aff$ such that $s= s_A$ and we let $B:= d_{F'}^{-1} A$. 
Note that if $F'$ is of type $\Bii$ then $B\in \Pi_\aff$ and $s_B=d_{F'} ^{-1}s_A d_{F'}$.
\\For (b), assume that $F$ is of type $\A$  and  $F'$  of type $\Bii$ for $s$. We have
  $( d_{F'}^{-1} d_F)^{-1} B\in \Phi^{-}_\aff$ so
 $\ell(s_B d_{F'}^{-1} d_F)=\ell(d_{F'}^{-1} d_F)-1$ and
 
 $T_{s_B} T_{d_{F'}^{-1} d_F}= (\a+\b)T_{d_{F'}^{-1} d_F}-\a\b T_{ s_Bd_{F'}^{-1} d_F}= (\a+\b)T_{d_{F'}^{-1} d_F}-\a\b T_{d_{F'} ^{-1}s_A d_F}= (\a+\b)T_{d_{F'}^{-1} d_F}-\a\b T_{d_{F'} ^{-1}d_{s_AF}}$.\\
 For $(c)$, assume that $F$ is of type $\Bi$ and  $F'$ of type $\Bii$ for $s$. Then    $( d_{F'}^{-1} d_F)^{-1} B\in \Phi^{+}_\aff$ so
$\ell(s_B d_{F'}^{-1} d_F)=1+\ell(d_{F'}^{-1} d_F)$ and 
$T_{s_B} T_{d_{F'}^{-1} d_F}=T_{d_{F'}^{-1} s_A d_F}=T_{d_{F'}^{-1} d_{s_AF}}$.\\
Lastly if $F$ (hence $F'$) is of type $\Bii$ for $s$, let $C:= d_F^{-1} A\in \Pi_\aff$ and $s_C:=d_F^{-1} s_A d_F$.  We have
$d_{F'}^{-1} d_{F} C\in  \Phi^+_\aff$ so
 $ \ell(d_{F'}^{-1} d_{F})
+1=\ell(d_{F'}^{-1} d_{F} s_C)=\ell(d_{F'}^{-1} s_A d_F)$ and 
$T_{d_{F'}^{-1} s_A d_F}=T_{d_{F'}^{-1} d_{F}} T_{s_C}
$. On the other hand $(d_{F'}^{-1} d_{F})^{-1} B\in \Phi_\aff^+$ so
$1+ \ell(d_{F'}^{-1} d_F)=\ell(s_Bd_{F'}^{-1} d_F)=\ell(d_{F'}^{-1} s _A d_F)$
and $T_{s_B} T_{d_{F'}^{-1} d_F}= T_{d_{F'}^{-1} s_A  d_F}$.
Hence  $T_{s_B} T_{d_{F'}^{-1} d_F}=T_{d_{F'}^{-1} d_{F}} T_{s_C}$.
\end{proof}

Obviously $r_F^F$ is the identity map. 
Furthermore, it follows from 
Proposition \ref{prop:C(F)}-ii that  $\ell(d_{F''}^{-1} d_{F'})+
\ell(d_{F'}^{-1} d_{F})=\ell(d_{F''}^{-1} d_{F})\ $ for facets $F,F',F''$  such that $F''\subseteq \overline F'\subseteq \overline F$, and therefore, using \eqref{braid}, we have
$r^{F'}_{F''}\circ  r^F_{F'}=r^F_{F''}$.    
We   may therefore consider the coefficient system of right $\Hh$-modules $\CH$ given by the following data: to each facet $F$  we attach   right $\Hh$-module  $\Hh$ and,
for facets $F$ and $F'$ satisfying $F'\subseteq \overline F$, we choose the
 transition maps $r^{F}_{F'}$ as defined in \eqref{f:defir}. 
The coefficient system $\CH$
 yields a complex of right $\Hh$-modules as in \eqref{f:apart-complexV}.
We define a right $\Hh$-equivariant augmentation map by
\begin{align*}\upalpha:  C_c^{or} (\mathscr{A}_{(0)}, \CH)\cong \bigoplus_{x\in \Aa_{(0)}} \Hh&\longrightarrow  \Hh\cr (h_x)_x&\longmapsto \sum_x T_{d_x} h_x .\end{align*}

\begin{theorem} 
The augmented complex
\begin{equation}\label{f:complexC}
    0 \longrightarrow C_c^{or} (\mathscr{A}_{(d)}, \CH) \xrightarrow{\;\partial\;} \ldots \xrightarrow{\;\partial\;} C_c^{or} (\mathscr{A}_{(0)}, \CH)  \xrightarrow{\upalpha}  \Hh\longrightarrow 0 \ 
\end{equation} 
is an exact complex of right $\Hh$-modules.
\end{theorem}

\begin{proof} 
We first verify $\upalpha\circ \partial=0$. Given a $1$-dimensional facet $F$ with vertices $x$ and $y$ we consider the $1$-chain with support $F$ and value $h\in \Hh$ at $(F, c_F)$. Its image by $\partial$ is, up to a sign, the $0$-chain supported by $\{x,y\}$
with values
$$x \mapsto T_{d_x ^{-1} d_F} h \text{ and } y\mapsto T_{d_y ^{-1} d_F} h \ .$$
The image of this $0$-chain by the augmentation map is
$ T_{d_x} T_{d_x ^{-1} d_F} h- T_{d_y}T_{d_y ^{-1} d_F} h$ (up to a sign)
which, by \eqref{braid} and Proposition \ref{prop:C(F)}-ii, is zero.

The proof of the exactness goes through exactly as in \cite[Theorem  3.4]{OS1}. The ingredients are the following.\\
1) Given a facet $F$ in $\Aa$, the transition map $r^{C(F)}_F$ is the identity map (Remark \ref{rema:locc}).\\
2) For $n\geq 1$ and $D$ a chamber at (gallery) distance $n$ of $C$, define 
$\Aa(n-1)$ to be the set of facets contained in the closure of the chambers at distance $\leq  n-1$ of $C$.
The simplicial subcomplexes $\Aa(n-1)$ and $\Aa(n-1)\cup\overline D$ are contractible. This is proved in \cite[Proposition 4.16]{OS1}.
\end{proof}

For $i \ge 0$ fix a (finite) subset of facets $\mathscr{F}_i \subset \Aa_i \cap \overline{C}$ representing the $W$-orbits of $\Aa_i$.

\begin{proposition}\label{prop:isoBimo}
Each $C_c^{or} (\mathscr{A}_{(i)}, \CH)$ carries the structure of an $\Hh$-bimodule such that:
 \begin{enumerate}
 \item this extends its natural structure as a right $\Hh$-module,
 \item $C_c^{or} (\mathscr{A}_{(i)}, \CH)\cong \bigoplus_{F\in \mathscr F_i}\Hh (j_F)\otimes_{\HF} \Hh$,
 \item the maps $\partial$ and $\upalpha$ in \eqref{f:complexC} are $\Hh$-biequivariant.
 \end{enumerate}
 \end{proposition}
\begin{proof}

Fix an orientation $(F_0, c_{F_0})$ for each facet in $\overline C$. For any facet $F$ of $\Aa$ we let $F_0:= d_F^{-1} F$ and we choose
 $(F, c_F):= d_F(F_0, c_{F_0}).$ 
 Given $w\in \Waff$, it is then easy to check that 
 \begin{equation} \label{f:invO} w(F, c_F)= (w F, c_{wF})\ .\end{equation}
 Given two orientations $(F, c)$ and $(F, c')$ of $F$ we let 
 $\delta^{(F,c)}_{(F,c')}:= 1$ if $(F,c)=(F, c')$ and  $\delta^{(F,c)}_{(F,c')}:= -1$ otherwise. 

\medskip

\noindent Given  a facet $F$ and $h\in \Hh$ we denote by $(F,c_F)\mapsto h$ the oriented chain with support $F$ and value $h$ at $(F, c_F)$.

 \medskip
 
We define an isomorphism of right $\Hh$-modules
\begin{align} \label{f:iso}C_c^{or} (\mathscr{A}_{(i)}, \CH)&\longrightarrow    \bigoplus_{G_0\in \mathscr F_i}\Hh(j_{G_0})\otimes_{\HGO} \Hh
\end{align} 
 as follows. Let  $F$ be a facet of dimension $i$. 
 For $h\in \Hh$, we consider the oriented $i$-chain 
 $\gamma\::\:(F,c_F)\mapsto h \ .$
 The facet $F_0:= d_F^{-1} F$ in $\overline C$ is $\Omega$-conjugate to a unique 
 $G_0\in \mathscr F_i$. We choose $\omega \in \Omega$ such that
 $F_0:= \omega G_0$.
We attach to  $\gamma$  the element
 $$\delta^{(F_0, c_{F_0})}_{\omega(G_0, c_{G_0})}\:\:T_{d_F} T_\omega \otimes T_{\omega^{-1}} h\in\Hh(j_{G_0})\otimes_{\HGO} \Hh \ .$$
 It is easy to see that this does not depend on the choice of $\omega$ because if  $\omega'\in \Omega$ is such that $\omega' G_0= F_0$, then 
 $u:=\omega^{-1}\omega'$ lies in $\Omega _{G_0}$ and by definition $\epsilon_{G_0}(u)=\delta^{u(G_0, c_{G_0})}_{(G_0, c_{G_0})}=\delta^{\omega'(G_0, c_{G_0})}_{\omega(G_0, c_{G_0})}$.

We define a map in the other direction. Given $G_0\in \Ff_i$,  we have $\Hh(j_{G_0})\otimes_{\HGO} \Hh\cong \oplus_{d\in \D_{G_0}^\dagger} T_d\otimes H$  (Proposition \ref{prop:frees}). Let $d\in \D_{G_0}^\dagger$, $h\in H$ and $F:= d G_0$. Then $\omega:= d_F^{-1} d$ lies in $\Omega$ (because $d^{-1} F$ and $d_F^{-1} F$ both are facets in $\overline C$). We attach to the element $T_d\otimes h\in \Hh(j_{G_0})\otimes_{\HGO} \Hh$ the oriented $i$ chain  $(F, c_F)\mapsto \delta^{(F, c_F)}_{d(G_0, c_{G_0})} T_\omega h$.
It is easy to check that this defines an inverse for \eqref{f:iso}.

Let $i\geq 0$. We endow $C_c^{or} (\mathscr{A}_{(i)}, \CH)$ with the structure of  $\Hh$-bimodule inherited from $\bigoplus_{F\in \mathscr F_i}\Hh(j_F)\otimes_{\HF} \Hh$, thus  extending its natural structure of right $\Hh$-module.
We describe the left action of $\Hh$ on  $C_c^{or} (\mathscr{A}_{(i)}, \CH)$.\\
Let  $F$ be a facet  and $F_0:= d_F^{-1} F$. For  $h\in \Hh$,
we consider the oriented $i$-chain 
 $\gamma\::\:(F,c_F)\mapsto h \ .$

\begin{itemize}
\item[\textbf{I.}] Let $\omega\in \Omega$. The action of $T_\omega$ maps $\gamma$ onto
the oriented $i$-chain 
 $T_\omega\cdot \gamma\::\:(\omega F,c_{\omega F})\mapsto\delta^{\omega(F_0, c_{F_0})} 
 _{(\omega F_0, c_{\omega F_0})}
 T_\omega  h \ .$
\item[\textbf{II.}] Let $s\in \Saff$.   We refer to Definition \ref{defi:types}.

\begin{itemize}
\item Suppose $F$ is of type $\A$ for $s$. The action of $T_s$  maps $\gamma$ onto the sum of oriented chains
\begin{align*} (F,c_F)\mapsto (\a+\b) h\quad \quad+\quad \quad
(sF, c_{sF})&\mapsto  -\a\b \,h \ .\end{align*}
\item Suppose $F$ is of type $\Bi$ for $s$. 
The action of $T_s$  maps $\gamma$ onto the  oriented chain
$
(sF, c_{sF})\mapsto  h.$\item Suppose $F$ is of type $\Bii$ for $s$. 
The action of $T_s$  maps $\gamma$ onto the oriented chain
$(F, c_F)\mapsto T_{d_F^{-1} s d_F} h$.
\end{itemize}
\end{itemize}

Before showing that $\upalpha$ and $\partial$ are  left $H$-equivariant, we recall the remark at the beginning  
of the proof of Lemma \ref{lemma:tech} and  add the following one: for  a facet $F$ and $\omega\in \Omega$, we have  
$d_{\omega F}= \omega d_F \omega^{-1}$.

\noindent{We show that $\upalpha$ is left $H$-equivariant.}
Let $F=x$ be a vertex and recall that $\upalpha(\gamma)= T_{d_x} h$. 
\begin{enumerate}
\item We have $d_{\omega x}= \omega d_x \omega ^{-1}$. Therefore $\upalpha(T_\omega\cdot \gamma)=  T_{ \omega d_x \omega ^{-1}} T_\omega  h=T_{ \omega}  T_{d_x}   h= T_\omega\upalpha(\gamma)$.

\item Let $s\in \Saff$. We refer to Definition \ref{defi:types}.
\begin{enumerate}
\item Suppose that $x$ is of type $\A$ for $s$. Then $sd_x= d_{sx}$ and
 \begin{align*}\upalpha(T_s\cdot \gamma)&=T_{d_x} (\a+\b) h-\a\b T_{s d_x} h= T_sT_{sd_x} (\a+\b) h-\a\b T_{s d_x} h=T_{s} \upalpha(\gamma)\ .\end{align*}
\item Suppose $\ell(sd_x)= \ell(d_x)+1$. 
\begin{enumerate}
\item if  $x$ is of type $\Bi$  then $d_{sx}= s d_x$ so
$  \upalpha(T_s\cdot \gamma)= T_{sd_{x}} h= T_s T_{d_x} h= T_s \upalpha(\gamma)\ .$
\item if  $x$ is of type $\Bii$  then $\upalpha(T_s\cdot \gamma)=  T_{d_x} T_{d_x^{-1} s d_x} h= T_{s d_x} h= T_s T_{d_x} h= T_s \upalpha(\gamma)\ . $
\end{enumerate}
\end{enumerate}
\end{enumerate}

\noindent {We show that $\partial$ is left $H$-equivariant.} Suppose $i\geq 1$. By definition, $\partial (\gamma)$ is the sum over the facets  $F'$ of dimension $i-1$ contained in $\overline F$ of the chains  $(F', c_{F'})\mapsto \delta^{(F, c_F)\vert_{F'}}_{(F', c_{F'})} T_{d_{F'}^{-1}d_F} h$ where  we recall that $(F, c_F)\vert_{F'}$ is  the facet $F'$ equipped with the orientation induced by $(F, c_F)$.
We let $G'_0:= d_{F'}^{-1} F'$. 
\begin{enumerate}
\item
The action of $T_\omega$ on $\partial(\gamma)$ gives the sum over the same $F'$s of 
$$(\omega F', c_{\omega F'})\mapsto\delta^{\omega(G'_0, c_{G'_0})}_{(\omega F_0', c_{\omega F_0'})} \delta^{(F, c_F)\vert_{F'}}_{(F', c_{F'})} T_\omega T_{d_{F'}^{-1}d_F}   h$$ 
while $\partial( T_\omega \cdot \gamma)$ is the sum over the $F'$s of 
$$(\omega F', c_{\omega F'})\mapsto\delta^{(\omega F, c_{\omega F})\vert_{\omega F'}}_{(\omega F', c_{\omega F'})}\delta^{\omega(F_0, c{F_0})}_{(\omega F_0, c_{\omega F_0})}  T_{ d_{\omega F'}^{-1} d_{\omega F}} T_{\omega } h= \delta^{(\omega F, c_{\omega F})\vert_{\omega F'}}_{(\omega F', c_{\omega F'})}\delta^{\omega(F_0, c{F_0})}_{(\omega F_0, c_{\omega F_0})}  T_{\omega }T_{ d_{F'}^{-1} d_{ F}}  h \ .
$$ 
We check that the two displayed formulas above are equal. Note that $u:=d_{F'}^{-1} d_F\in W_{F_0'}$.
So
$$\delta^{\omega(G'_0, c_{G'_0})}_{(\omega F_0', c_{\omega F_0'})} \delta^{(F, c_F)\vert_{F'}}_{(F', c_{F'})}= \delta^{\omega(G'_0, c_{G'_0})}_{(\omega F_0', c_{\omega F_0'})} \delta^{d_F(F_0, c_{F_0})\vert_{F'}}_{d_{F'}(F_0', c_{F_0'})}=
\delta^{\omega(G'_0, c_{G'_0})}_{(\omega F_0', c_{\omega F_0'})} \delta^{u(F_0, c_{F_0})\vert_{F_0'}}_{(F_0', c_{F_0'})}=\delta^{u(F_0, c_{F_0})\vert_{F_0'}}_{\omega^{-1}(\omega F_0', c_{\omega F_0'})}= \delta^{\omega u(F_0, c_{F_0})\vert_{\omega F_0'}}_{(\omega F_0', c_{\omega F_0'})}$$
while, if we let $(F_0, \kappa_{F_0}):= \omega^{-1}( \omega F_0, c_{\omega F_0})$,  
$$\delta^{(\omega F, c_{\omega F})\vert_{\omega F'}}_{(\omega F', c_{\omega F'})}\delta^{\omega(F_0, c_{F_0})}_{(\omega F_0, c_{\omega F_0})} = \delta^{\omega u \omega^{-1}(\omega F_0, c_{\omega F_0})\vert_{\omega F_0'}}_{(\omega F_0', c_{\omega F_0'})}\delta^{\omega(F_0, c_{F_0})}_{(\omega F_0, c_{\omega F_0})}=\delta^{\omega u (F_0, \kappa_{ F_0})\vert_{\omega F_0'}}_{(\omega F_0', c_{\omega F_0'})}\delta^{(F_0, c_{F_0})}_{(F_0, \kappa_{F_0})}=\delta^{\omega u (F_0, c_{ F_0})\vert_{\omega F_0'}}_{(\omega F_0', c_{\omega F_0'})}\ .$$

\item 
 Now we consider the action of $T_s$ for $s\in \Saff$. 

$\bullet$
Suppose that $F$ is of type  \textbf{A}  for $s$. Recall (using \eqref{f:invO}) that $\partial (T_s\cdot \gamma)$ is the sum over the facets $F'$ of codimension $1$ in $\overline F$ of the chains
$
(F', c_{F'})\mapsto \delta^{(F, c_F)\vert_{F'}}_{(F', c_{F'})} (\a+\b) r^F_{F'}(h)$\ and 
$(sF', c_{sF'})\mapsto -\delta^{(F, c_{F})\vert_{F'}}_{(F', c_{F'})} \a\b \, r^{sF}_{sF'}( h). $

To see that it is equal to $T_s\cdot \partial (\gamma)$, we 
use Lemma \ref{lemma:tech}(a),(b) and we write it as the sum  of 
\begin{align*}
(F', c_{F'}) \mapsto \delta^{(F, c_F)\vert_{F'}}_{(F', c_{F'})} (\a+\b)r^{F}_{F'}( h)\text{ and }
(sF', c_{sF'}) \mapsto -\delta^{(F, c_{F})\vert_{F'}}_{(F', c_{F'})} \a\b r^{F}_{F'}( h)
\end{align*}where $F'$ ranges over the facets $F'$ of codimension $1$ in  $\overline F$ which are  of type $\textbf{A}$  for $s$
and of 
\begin{align*}
(F', c_{F'})&\mapsto \delta^{(F, c_F)\vert_{F'}}_{(F', c_{F'})} [(\a+\b) r^{F}_{F'}( h)-
\a\b r^{sF}_{F'}( h) ]=\delta^{(F, c_F)\vert_{F'}}_{(F', c_{F'})} T_{d_{F'}^{-1} s d_{F'}} r^F_{F'}(h)
\end{align*}
where $F'$ ranges over the facets $F'$ of codimension $1$ in $\overline F$ which are  of type $\Bii$. 

$\bullet$ Suppose that $F$  of type $\Bi$ for $s$.
Here $\partial (T_s\cdot \gamma)$ is the sum over the facets $F'$ of codimension $1$ in $\overline F$ of the chains 
$(sF', c_{sF'})\mapsto \delta^{(F, c_F)\vert_{F'}}_{(F', c_{F'})} r^{sF}_{sF'}( h)  \ .$
 To see that it is equal to $T_s\cdot \partial (\gamma)$, we 
use Lemma \ref{lemma:tech}(a),(c) and  write it as  the sum of the chains
\begin{align*}
(sF', c_{sF'})&\mapsto \delta^{(F, c_F)\vert_{F'}}_{(F', c_{F'})} r^{F}_{F'}( h) \ .
\end{align*}
where $F'$ ranges over the facets $F'$ of codimension $1$ in $\overline F$ which are  of type $\Bi$ 
and of 
\begin{align*}
(F', c_{F'})&\mapsto \delta^{(F, c_F)\vert_{F'}}_{(F', c_{F'})} T_{d_{F'}^{-1} s d_{F'}} r^F_{F'}(h) \ .
\end{align*}
where $F'$ ranges over the facets $F'$ of codimension $1$ in $\overline F$ which are  of type $\Bii$.

$\bullet$
Lastly suppose that $F$ is of type $\Bii$.
Here $\partial (T_s\cdot \gamma)$ is the sum over the facets $F'$ of codimension $1$ in $\overline F$ of the chains (using \eqref{f:invO})
$(F', c_{F'})\mapsto \delta^{(F, c_F)\vert_{F'}}_{(F', c_{F'})} r^F_{F'}(T_{d_{F}^{-1} s d_{F}} h) \ .$
To see that it is equal to $T_s\cdot \partial (\gamma)$, we 
use  Lemma \ref{lemma:tech}(d) and  write it as  the sum over the facets $F'$ of codimension $1$ in $\overline F$ of the chains 
$(F', c_{F'})\mapsto \delta^{(F, c_F)\vert_{F'}}_{(F', c_{F'})} 
T_{d_{F'}^{-1} s d_{F'}}   r^F_{F'}(h) \  .$
\end{enumerate}\end{proof}  
  
\begin{corollary} \label{cor:main1}
We have an exact resolution of $\Hh$ by $\Hh$-bimodules
\begin{equation}\label{f:complexH}
0 \longrightarrow \bigoplus_{F\in \mathscr F_d}\Hh(j_F)\otimes_{\HF} \Hh \longrightarrow \ldots \longrightarrow\bigoplus_{F\in \mathscr F_0}\Hh(j_F)\otimes_{\HF} \Hh\longrightarrow  \Hh\longrightarrow 0.
\end{equation}  
Moreover, each term in this resolution is free as a left (resp. right) $\Hh$-module.
\end{corollary}

For later use we record an explicit description of the first (left) differential in (\ref{f:complexH}). For simplicity, fix an orientation $(C, c_C)$ of $C$ and for each  of its codimension $1$ facet $F$ choose the orientation $(F, c_{F}):=(C, c_C)\vert_{F}$. In this case one checks (using the explicit form  of the isomorphism  \eqref{f:iso}, also compare with \cite[Cor. 6.7]{OS1}) that the first differential in (\ref{f:complexH}) is given by 
\begin{equation}\label{f:diffC}
\partial\quad :\quad \Hh(j_C)\otimes_{\HC} \Hh \ni\quad  1 \otimes 1 \longmapsto  \sum_{F\in \mathscr F_{d-1}}\sum_{\omega\in \Omega/\Omega_{F}} j_C(T_\omega) \otimes T_{\omega^{-1}} \quad \in \bigoplus_{F\in \mathscr F_{d-1}}\Hh(j_F)\otimes_{\HF} \Hh.
\end{equation}

\section{Dualizing complexes}
 
In this section we work over a fixed ground field $\k$. We assume that all rings are left and right Noetherian algebras over $\k$. For a $\k$-algebra $A$ we denote by $D^b(A)$ the bounded derived category of finitely generated $A$-modules.

\subsection{Rigid dualizing complexes}\label{sec:rigiddualizing}

The following definition follows \cite{Ye1}. 

\begin{definition}\label{def:dualizing1}
An object $\R \in D^b(A \otimes_\k A^o)$ is called a dualizing if 
\begin{enumerate}
\item $\R$ has finite injective dimension over $A$ and $A^o$,
\item the cohomology of $\R$ is given by bimodules which are finitely generated on both sides,
\item the natural morphisms $A \to \RHom_A(\R,\R)$ and $A \to \RHom_{A^o}(\R,\R)$ are isomorphisms in $D(A \otimes_\k A^o)$.
\end{enumerate}
More generally, suppose $\z  \subset A$ is a finitely generated,  commutative, central  $k$-subalgebra    so that $A$ is flat over $\z $. Then $\R \in D^b(A \otimes_\z  A^o)$ is called dualizing if its restriction to $D^b(A \otimes_k A^o)$ is dualizing. 
\end{definition}

 \begin{remark}
We could weaken the assumption that $A$ is flat over $\z $ to $A$ having finite tor-dimension over $\z $ but at the cost of working with dg-algebras. 
\end{remark}

\begin{remark}
A dualizing complex $\R \in D^b(A \otimes_\k A^o)$ induces equivalences 
$$\RHom_A(-,\R): D^b(A) \leftrightarrows D^b(A^o): \RHom_{A^o}(-,\R)$$
which explains the terminology 	``dualizing complex'' (see \cite[Prop. 3.5]{Ye1}). 
\end{remark}

\begin{definition}\label{def:dualizing2}
Let $\z  \subset A$ be as in Definition \ref{def:dualizing1}. Then $\R \in D^b(A \otimes_\z  A^o)$  is $\z $-rigid if there is an isomorphism 
$$\phi: \R \rightarrow \RHom_{A \otimes_\z  A^o}(A, \R \otimes_\z  \R)$$
in $D(A \otimes_\z  A^o)$. If $\z =\k$ then we say that $\R$ is rigid. 
\end{definition}

\begin{remark}
The definition above appears in \cite[Def. 8.1]{Ber} when $\z  = \k$. We will need the mild generalization above for our applications. It is shown in \cite{Ber} that, if they exist, rigid dualizing complexes are unique. The same argument also shows that $\z $-rigid dualizing complexes are unique (when $\z $ is regular this also follows from Lemma \ref{lem:dualizing} below). When it exists we will denote the rigid (resp. $\z $-rigid) dualizing complex by $\R_A$ (resp. $\R_{A/\z })$. 
\end{remark}

If $\pi_\z : \Spec \z  \to \Spec \k$ is the natural map then $\R_\z  := \pi_\z ^!(\k) \in D(\z )$ is rigid when viewed as a $\z $-bimodule. Here $\pi_\z ^!$ is the usual twisted inverse image functor from Grothendieck duality. Similarly, if $A$ is commutative then $\R_A := \pi_A^!(\k) \in D(A)$ is rigid (where $\pi_A: \Spec A  \to \Spec \k$) and $\R_{A/\z } := \pi^!(\z ) \in D(A)$ is $\z $-rigid where $\pi: \Spec A \to \Spec \z $.  Thus Definition \ref{def:dualizing2} is an attempt to identify this canonical relative dualizing object when $A$ is not commutative. 

Since $\pi$ is flat it also follows that $\pi^!(\R_\z ) \cong \pi^!(\z ) \otimes_\z  \R_\z $ (see for example \cite[Thm. 4.9.4]{LH}). In other words, $\R_A \cong \R_{A/\z } \otimes_\z  \R_\z $. The following result is a non-commutative analogue of this observation. 

\begin{lemma}\label{lem:dualizing}
Let $\z  \subset A$ be as in Definition \ref{def:dualizing1} and suppose $\z $ is regular with rigid dualizing complex $\R_\z $. Then $A$ has a rigid dualizing complex $\R_A$ if and only if it has $\z $-rigid dualizing complex $\R_{A/\z }$. In this case 
\begin{equation}\label{eq:RAC}
\R_A \cong \R_{A/\z } \otimes_\z  \R_\z  \in D(A \otimes_\z  A^o).
\end{equation}
\end{lemma}
\begin{proof}
Since $\z $ is Noetherian and  regular its rigid dualizing complex $\R_\z $ is invertible.  Thus, for $\R_A$ and $\R_{A/\z }$ satisfying (\ref{eq:RAC}), it follows that $\R_A$ is dualizing if and only if $\R_{A/\z }$ is dualizing. It remains to show that $\R_A$ is rigid if and only if $\R_{A/\z }$ is $\z $-rigid. 

To see this, consider the following sequences of isomorphisms
\begin{align*}
\RHom_{A \otimes_\k A^o}(A, \R_A \otimes_\k \R_A) 
&\cong \RHom_{A \otimes_\z  A^o}(A \otimes_{A \otimes_\k A^o} (A \otimes_\z  A^o), (\R_{A/\z } \otimes_\z  \R_\z ) \otimes_\k (\R_{A/\z } \otimes_\z  \R_\z )) \\
&\cong \RHom_{A \otimes_\z  A^o}(A \otimes_{\z  \otimes_\k \z } \z , (\R_\z  \otimes_\k \R_\z ) \otimes_{\z  \otimes_\k \z } (\R_{A/\z } \otimes_\k \R_{A/\z })) \\
&\cong \RHom_{A \otimes_\z  A^o}(A, \RHom_{\z  \otimes_\k \z }(\z , (\R_\z  \otimes_\k \R_\z ) \otimes_{\z  \otimes_\k \z } (\R_{A/\z } \otimes_\k \R_{A/\z })))\\
&\cong \RHom_{A \otimes_\z  A^o}(A, \RHom_{\z  \otimes_\k \z }(\z , \R_\z  \otimes_\k \R_\z ) \otimes_{\z  \otimes_\k \z } (\R_{A/\z } \otimes_\k \R_{A/\z })) \\
&\cong \RHom_{A \otimes_\z  A^o}(A, \R_\z  \otimes_{\z  \otimes_\k \z } (\R_{A/\z } \otimes_\k \R_{A/\z })) \\
&\cong \RHom_{A \otimes_\z  A^o}(A, \R_\z  \otimes_\z  (\R_{A/\z } \otimes_\z  \R_{A/\z })) \\
&\cong \RHom_{A \otimes_\z  A^o}(A, \R_{A/\z } \otimes_\z  \R_{A/\z }) \otimes_\z  \R_\z  
\end{align*}
Here the first isomorphism is by adjunction of restriction and induction, the second is by base change noting that 
$$A \otimes_\z  A^o \cong (A \otimes_\k A^o) \otimes_{\z  \otimes_\k \z } \z $$
the third is by adjunction between tensor product and hom, the fourth uses that $\z $ is perfect inside $D(\z  \otimes_\k \z )$ because $\z $ is regular, the fifth uses that $\R_\z $ is rigid, the sixth uses that the left and right actions of $\z $ on $\R_\z $ agree and the last uses that $\R_\z $ is (up to shift) a locally free $\z $-module since $\z $ is regular. Since $\R_\z $ is invertible,  the result follows. 
\end{proof}

\subsection{Traces}

\begin{definition}\label{def:trace}
Let $f: A \to B$ be a finite morphism of $\k$-algebras and suppose $A$ and $B$ have rigid dualizing complexes $(\R_A,\phi_A)$ and $(\R_B,\phi_B)$. Then $\tr_{B/A}: \R_B \to \R_A$ in $D(A \otimes_\k A)$ is called a trace morphism if the following conditions hold:
\begin{enumerate}
\item $\tr_{B/A}$ induces an isomorphism in $D(A \otimes_\k A)$
\begin{equation}\label{eq:trace1}
\R_B \cong \RHom_A(B,\R_A) \cong \RHom_{A^o}(B,\R_A)
\end{equation}
\item the following diagram in $D(A \otimes_\k A)$ commutes
\begin{equation}\label{eq:trace2}
\xymatrix{
\R_B \ar[rrr]^(.35){\phi_B}  \ar[d]^{\tr_{B/A}} & & & \RHom_{B \otimes_\k B}(B, \R_B \otimes_\k \R_B) \ar[d]^{\tr_{B/A} \otimes  \tr_{B/A}} \\
\R_A \ar[rrr]^(.35){\phi_A} & & & \RHom_{A \otimes_\k A}(A, \R_A \otimes_\k \R_A)
}
\end{equation}
\end{enumerate}
\end{definition}

If they exist, trace morphisms $\tr_{B/A}$ are unique \cite[Thm. 3.2]{YZ1}. A consequence of traces is the following duality result. It is a non-commutative counterpart of the result in algebraic geometry which states that for a proper morphism $f: X \to Y$ of schemes of finite type one has 
$$f_* \RHom_X(M, f^! (N)) \cong \RHom_Y(f_* M, N)$$
for coherent sheaves $M,N$ on $X,Y$ respectively. 

\begin{corollary}[Prop. 3.9(1) \cite{YZ1}]\label{cor:dual1} 
Suppose we are in the setup of Definition \ref{def:trace}. Then $\tr_{B/A}$ induces a natural isomorphism 
\begin{equation}\label{eq:duality1}
\Res^{B^o}_{A^o} \circ \RHom_B(-,\R_B) \cong \RHom_A(\Res^B_A(-), \R_A): D_f(B) \to D_f(A^o).
\end{equation}
\end{corollary}

\subsection{Existence results}\label{sec:existence}

There are various results which guarantee the existence of rigid dualizing complexes. We highlight the ones which are relevant to our situation. 

\begin{definition}\label{def:differentialk}
Following \cite{YZ2} we say that $A$ is a differential $\k$-algebra of finite type if it has an exhaustive filtration $\{F_i A\}_{i \in \Z}$ such that the associated graded $\gr^F(A)$ is a finite module over its center which is a finitely generated $k$-algebra. 
\end{definition}

A nice consequence of \cite[Thm. 3.1]{YZ2} is that any differential $\k$-algebra $A$ of finite type has a rigid dualizing complex $\R_A$ \cite[Thm. 8.1]{YZ2}. More generally, if $A \to B$ is a finite centralizing homomorphism of $\k$-algebras then $B$ also has a rigid dualizing complex $\R_B$ and moreover there exists a trace morphism $\tr_{B/A}: \R_B \to \R_A$ \cite[Thm. 6.17]{YZ1}. Here $A \to B$ is finite centralizing if there exists a finite set $\{b_i\} \subset B$ commuting with $A$ such that $B = \sum A \cdot b_i$. In particular, this implies the following result.

\begin{proposition}\label{prop:finite/Z}
Suppose $A$ is finite over a central subalgebra $Z \subset A$ which is a finitely generated $k$-algebra. Then there exist rigid dualizing complexes $\R_A$ and $\R_Z$ with $\R_A \cong \RHom_Z(A, \R_Z)$. Moreover, there exists a trace map $\tr_{A/Z}: \R_A \to \R_Z$. If $Z$ is also regular then $\R_{A/Z} \cong \RHom_Z(A,Z)$. 
\end{proposition}
\begin{proof}
Since $Z \to A$ is finite centralizing we know there exist rigid dualizing complexes $\R_A$ and $\R_Z$ as well as the trace map $\tr_{A/Z}$. The isomorphism  $\R_A \cong \RHom_Z(A,\R_Z)$ follows from the proof of \cite[Prop. 5.9]{Ye2}. Finally, if $Z$ is regular, then $\R_Z$ is invertible and $\RHom_Z(A, \R_Z) \cong \RHom_Z(A,Z) \otimes_Z \R_Z$. Thus the last isomorphism follows from Lemma \ref{lem:dualizing}. 
\end{proof}

We say $A$ is Gorenstein if $A$ itself is a dualizing complex (cf. \cite[p. 68]{Ye1}). In this case any dualizing complex is invertible (\cite[Thm. 3.9]{Ye1}). 

\begin{proposition}\label{prop:gorenstein-dualizing}
Consider $\z  \subset A$ as in Definition \ref{def:dualizing1} where $\z $ is furthermore regular and $A$ is Gorenstein. If $A$ has a $\z$-rigid dualizing complex $\R_{A/\z}$ then 
\begin{equation}\label{eq:RAC2}
\R_{A/\z}^{-1} = \RHom_{A \otimes_\z  A^o}(A,A \otimes_\z  A)
\end{equation}  
where the action of $A \otimes_\z  A^o$ is via the outer action on $A \otimes_\z  A$. 
\end{proposition}

\begin{proof}
This is proved in \cite[Prop. 8.4]{Ber} when $\z  = \k$ but essentially the same argument works more generally.  We reproduce it here for completeness. Using that $\R_A$ and $\R_\z $ are invertible we find, using Lemma \ref{lem:dualizing}, that so is $\R_{A/\z }$. Then we get
\begin{align*}
\RHom_{A \otimes_\z  A^o}(A, A \otimes_\z  A^o) 
& \cong \RHom_{A \otimes_\z  A^o}(A, \R_{A/\z } \otimes_A \R_{A/\z }^{-1} \otimes_\z  \R_{A/\z }^{-1} \otimes_{A^o} \R_{A/\z }) \\
& \cong \RHom_{A \otimes_\z  A^o}(A, (\R_{A/\z } \otimes_\z  \R_{A/\z }) \otimes_{A \otimes_\z  A^o} (\R_{A/\z }^{-1} \otimes_\z  \R_{A/\z }^{-1})) \\
& \cong \RHom_{A \otimes_\z  A^o}(A, \R_{A/\z } \otimes_\z  \R_{A/\z }) \otimes_{A \otimes_\z  A^o} (\R_{A/\z }^{-1} \otimes_\z  \R_{A/\z }^{-1}) \\
& \cong \R_{A/\z } \otimes_{A \otimes_\z  A^o} (\R_{A/\z }^{-1} \otimes_\z  \R_{A/\z }^{-1}) \\
& \cong \R^{-1}_{A/\z } \otimes_A \R_{A/\z } \otimes_A \R_{A/\z }^{-1} \cong \R_{A/\z }^{-1}. 
\end{align*}
\end{proof}

\subsection{Base change}

Let $\z  \subset A$ be as in Definition \ref{def:dualizing1} and $\z'$ a  commutative, finitely generated $k$-algebra. Our running assumption in this section is that $\z,\z'$ are regular. Consider the base change $A' := A \otimes_\z \z'$ with respect some homomorphism $\z \to \z'$.
\begin{lemma}\label{lem:basechange}
If $\R \in D^b(A \otimes_\z A^o)$ has finite injective dimension over $A$ (resp. $A^o$) then $\RHom_{A \otimes_\z A^o}(A' \otimes_{\z'} A', \R)$ has finite injective dimension over $A'$ (resp. $A'^o$). 
\end{lemma}
\begin{proof}
For any $N' \in D^b(A')$ we have
\begin{align*}
\RHom_{A'}(N', \RHom_{A \otimes_\z A^o}(A' \otimes_{\z'} A', \R))
& \cong \RHom_{A \otimes_\z A^o}((A' \otimes_{\z'} A') \otimes_{A'} N', \R) \\
& \cong \RHom_{A \otimes_\z A^o}(A' \otimes_{\z'} N', \R) \\
& \cong \RHom_{A \otimes_\z A^o}(A \otimes_{\z} N', \R) \\
& \cong \RHom_{A \otimes_\z A^o}((A \otimes_{\z} A^o) \otimes_A N', \R) \\
& \cong \RHom_A(N', \R) 
\end{align*}
Thus, if $\R$ has finite injective dimension over $A$ then $\RHom_{A \otimes_\z A^o}(A' \otimes_{\z'} A', \R)$ has finite injective dimension over $A'$. The case of $A^o$ and $A'^o$ is similar. 
\end{proof}

\begin{proposition}\label{prop:basechange}
If $\R$ is a dualizing complex of $A$ then $\R \otimes_{\z} \z'$ is a dualizing complex of $A'$. 
% If $A$ is Gorenstein then so is $A'$. Moreover, if $\R_{A/\z}$ is the $\z$-rigid dualizing complex of $A$ then $\R_{A'/\z'} := \R_{A/\z} \otimes_\z \z'$ is the $\z'$-rigid dualizing complex of $A'$. 
\end{proposition}
\begin{proof}
We need to check that $\R' := \R \otimes_{\z} \z'$ satisfies the three conditions of Definition \ref{def:dualizing1}. Condition (2) is easy to see since $\z'$ has finite tor-dimension over $\z$ (since they are both regular rings). Condition (3) follows from the commutativity of the following rectangle
$$\xymatrix{
A'  \ar[rrr]^\sim \ar[d] & & & A \otimes_\z \z' \ar[d] \\
\RHom_{A'}(\R',\R') \ar[r]^-\sim & \RHom_{A'}(\R \otimes_\z \z', \R \otimes_\z \z') \ar[r]^-\sim & \RHom_A(\R, \R \otimes_\z \z') \ar[r]^-\sim & \RHom_A(\R, \R) \otimes_\z \z'
}$$

It remains to prove condition (1), namely that $\R'$ has finite injective dimension over $A'$ and $A'^o$. We first reduce to the case $\z'$ is finite over $\z$. To do this it suffices to check condition (1) when $\z' \cong \z[x]$. This follows since if $\R$ has injective dimension $e$ over $A$ (resp. $A^o$) then $\R \otimes_{\z} \z[x]$ has injective dimension $e+1$ over $A \otimes_{\z} \z[x]$ (resp. $A^o \otimes_\z \z[x]$). 

Thus we can assume $\z'$ is finite over $\z$. In this case, 
\begin{align*}
\RHom_{A \otimes_\z A^o}(A' \otimes_{\z'} A', \R) 
&\cong \RHom_{A \otimes_\z A^o}(A \otimes_\z A \otimes_\z \z', \R) \\
&\cong \RHom_{A \otimes_\z A^o}(A \otimes_\z A, \R) \otimes_\z \RHom_\z(\z',\z) \\
&\cong \R \otimes_\z \RHom_\z(\z',\z).
\end{align*}
By Lemma \ref{lem:basechange} it follows that $\R \otimes_\z \RHom_\z(\z',\z)$ has finite injective dimension over $A'$ and $A'^o$. Since $\z'$ is finite over $\z$ with both regular rings it follows that $\RHom_\z(\z',\z) \cong \R_{\z'/\z}$ is invertible as a $\z'$-module and thus $\R \otimes_\z \z'$ also has finite injective dimension over $A'$ and $A'^o$. This completes the proof. 
\end{proof}

\begin{corollary}\label{cor:basechange}
If $\R_{A/\z}$ is the $\z$-rigid dualizing complex of $A$ then $\R_{A/\z} \otimes_\z \z'$ is the $\z'$-rigid dualizing complex of $A'$. 
\end{corollary}
\begin{proof}
By Proposition \ref{prop:basechange} we know $\R_{A'/\z'} := \R_{A/\z} \otimes_\z \z'$ is a dualizing complex so it remains to show that it is $\z'$-rigid. Note that
\begin{equation}\label{eq:local3}
A \otimes_{A \otimes_\z A} (A' \otimes_{\z'} A') 
\cong A \otimes_{A \otimes_\z A} ((A \otimes_\z \z') \otimes_{\z'} (A \otimes_\z \z')) 
\cong A \otimes_{A \otimes_\z A} (A \otimes_\z A \otimes_\z \z') 
\cong A \otimes_\z \z' \cong A'.
\end{equation}
Thus we get
\begin{align*}
\RHom_{A' \otimes_{\z'} A'}(A', \R_{A'/\z'} \otimes_{\z'} \R_{A'/\z'}) 
& \cong \RHom_{A' \otimes_{\z'} A'}(A', (\R_{A/\z} \otimes_\z \z') \otimes_{\z'} (\R_{A/\z} \otimes_\z \z')) \\
& \cong \RHom_{A' \otimes_{\z'} A'}(A \otimes_{A \otimes_\z A} (A' \otimes_{\z'} A') , \R_{A/\z} \otimes_\z \R_{A/\z} \otimes_\z \z') \\
& \cong \RHom_{A \otimes_{\z} A}(A, \R_{A/\z} \otimes_\z \R_{A/\z} \otimes_\z \z') \\
& \cong \RHom_{A \otimes_{\z} A}(A, \R_{A/\z} \otimes_\z \R_{A/\z}) \otimes_\z \z' \\
& \cong \R_{A/\z} \otimes_\z \z' \cong \R_{A'/\z'}
\end{align*}
where the second isomorphism is by rearranging and using (\ref{eq:local3}), the third is by adjunction between induction and restriction, the fourth uses that $\z$ and $\z'$ are regular, the fifth is because $\R_{A/\z}$ is $\z$-rigid and the last is by definition. This proves that $\R_{A'/\z'}$ is $\z'$-rigid. 
\end{proof}

\subsection{Further consequences}

\begin{proposition}\label{prop:miracle}
Suppose $A$ is finite over a connected, regular, central subalgebra $Z \subset A$ which is   finitely generated as a  $k$-algebra. Then $\R_A$ is supported in one degree (i.e. $A$ is Cohen-Macaulay) if and only if $A$ is projective over $Z$. In this case $\R_A$ is supported in degree $-d$ where $d$ is the Krull dimension of $Z$. 
\end{proposition}
\begin{proof}
Note that by Proposition \ref{prop:finite/Z} we know $A$ has a rigid dualizing complex $\R_A$. 
Since $Z$ is regular we have 
\begin{equation}\label{eq:local2}
\RHom_Z(A,M) \cong \RHom_Z(A,Z \otimes_Z M) \cong \RHom_Z(A,Z) \otimes_Z M
\end{equation}
for any $M \in D^b(Z)$.  In particular, taking $M = \R_Z$, this gives  
$$\RHom_Z(A,Z) \cong \RHom_Z(A,\R_Z) \otimes_Z \R_Z^{-1} \cong \R_A \otimes_Z \R_Z^{-1}.$$
Thus, if $\R_A$ is Cohen-Macaulay, then $\RHom_Z(A,Z)$ is supported in one degree. But $Z \hookrightarrow A$ and after localizing this map splits. It follows that $\RHom_Z(A,Z)$ must be supported in degree zero. 

On the other hand, $\Hom_Z(A,-)$ is left exact and tensoring is right exact so (\ref{eq:local2}) must be supported in degree zero for any module $M$. In particular, this means $\Hom_Z(A,-)$ is exact and thus $A$ is projective over $Z$. 

Conversely, if $A$ is projective over $Z$ then $\RHom_Z(A,\R_Z) \cong \R_A$ is supported in the same degree as $\R_Z$. Since $\R_Z$ is supported in degree $-d$ the result follows.   
\end{proof}

\begin{remark}
The result of Corollary \ref{prop:miracle} when $A$ is commutative is called Hironaka's criterion (or miracle flatness). 
\end{remark}

\begin{proposition}\label{prop:fin-module}
Suppose $A$ is a differential $\k$-algebra of finite type and $M$ a finite dimensional $A$-module. Then
\begin{equation}\label{eq:duality2}
\RHom_A(M, \R_A) \cong M^\vee \in D(A^o)
\end{equation}
where $\R_A$ is the rigid dualizing complex of $A$. \end{proposition}
\begin{proof}
The following argument follows the one from \cite[Cor.  2.2]{Ye3} (we thank Amnon Yekutieli for pointing out his result).  Let $B := A/I$ where $I$ is the kernel of the canonical map $A \to \End_\k(M)$. Note that $A \to B$ is surjective and hence finite centralizing. It follows by \cite[Thm. 6.17]{YZ1} that there exists a trace map $\tr_{B/A}$. 

By construction we have a canonical $M' \in D(B)$ with $M = \Res^B_A(M')$. Thus, using (\ref{eq:duality1}), we get that
\begin{equation}\label{eq:local1}
\Res^{B^o}_{A^o} \RHom_B(M',\R_B) \cong \RHom_A(M,\R_A).
\end{equation}
On the other hand, since $B$ is finite dimensional over $\k$, $\R_B \cong B^\vee = \Hom_\k(B,\k)$. It is standard to check that $\Hom_B(M',\Hom_\k(B,\k)) \cong \Hom_\k(M', \k)$. Since $\Res^{B^o}_{A^o}(\Hom_\k(M', \k)) \cong M^\vee$ the required isomorphism (\ref{eq:duality2}) follows from (\ref{eq:local1}). 
\end{proof}

\section{\label{sec:rigid}Rigid dualizing complexes of affine Hecke algebras}

In this section we assume $\r$ is a regular,  finitely generated $k$-algebra. This condition on $\r$ is only used to ensure that $\z = \r[\a,\b][Q^\perp] \subset \Hh$ is also regular. 

\subsection{The rigid dualizing complex of $\HF$}
Let $F$ be a facet in $\overline C$ and $\HF$ the associated Hecke algebra as in \S\ref{subsec:finiteH}. Let $w_0$ be the longest element of the finite Weyl  group $\WFf$ and define the $\r$-linear  involution 
\begin{equation*}
i_F: \HF \longrightarrow  \HF, \quad 
T_w  \longmapsto  T_{w_0 w w_0^{-1}} \ .
\end{equation*}
As usual $(i_F)\HF$ is the $\HF$-bimodule $\HF$ with the left action twisted by $i_F$.

\begin{lemma} \label{lemma:HF}\phantomsection
\begin{enumerate}
\item The map $i_F$ is an algebra automorphism which acts trivially on $\z$.
\item There exists an isomorphism  {of $\HFz$-modules}
\begin{equation}\label{eq:HF}
 (i_F)\HF \cong \RHom_{\z}(\HF, \z)
 \end{equation}
\end{enumerate}
\end{lemma}
\begin{proof}
From Proposition  \ref{prop:freesF} and its proof, $\HF$ is free over $\z$ with basis $\{T_w\}_{w \in [\WF/{Q^\perp}]} = \{T_{w^{-1} w_0}\}_{w \in  [\WF/{Q^\perp}]}$. Consider, as in \cite[Prop. 5.4-iii]{OS1}, the $\z$-linear map
\begin{equation*}\theta:\:\HF\longrightarrow \z, \quad\sum_{w\in W_F^\dagger} a_w T_w\longmapsto \sum_{\xi\in {Q^\perp}} a_{\xi w_0} \xi \ .\end{equation*}
The same arguments as in \loccit ensure that the matrix $[\theta(T_w T_{v^{-1} w_0})]_{v,w\in[\WF/{Q^\perp}] }$ with coefficients in $\z$ is invertible. This means that the homomorphism of right $\HF$-modules
\begin{equation}\label{traceHF}\HF\longrightarrow \Hom_{\z}(\HF, \z),\quad  \ 1\longmapsto \theta\end{equation} is an isomorphism.
We are going to check that $\theta(i_F(x)_-)=\theta(_- x) $ for any $x\in \HF$.
Given that \eqref{traceHF} is bijective, this identity will imply that $i_F$ is an algebra automorphism while also proving (\ref{eq:HF}).  
To prove the identity,  we show for $w\in \WF$  that
 \begin{equation}\label{f:HFbim0}\theta(i_F(T_w)T_v)=\theta( T_v T_w) \text{ for any $v\in \WF$} \ .\end{equation}
By induction  it is enough to verify this for $\ell(w)\leq 1$.

$\bullet$  An element $\omega \in \Omega_F$  normalizes $\WFf$ and since the longest element of $\WFf$ is unique, we have $\omega w_0\omega^{-1}= w_0$.  Furthermore, since $w_0 \Phi_F^+= \Phi_F^-$, we have  $w_0\omega w_0^{-1}\in \Omega_F$.  This  allows  to check \eqref{f:HFbim0} when $w=\omega$.

$\bullet$ Now assume $w$ has length $1$.  Write $w=  s \omega$ for $s\in S_F$ and $\omega \in \Omega_F$.
 Since $w_0 \Pi_F= -\Pi_F$ there is $s'\in S_F$   and $\omega'\in \Omega_F$ such that 
 $w':=w_0 w w_0^{-1}$ can be written as   $\omega' s' $.
 Let $v\in \WF$.   Note that $vw\in {Q^\perp} w_0$ if and only if $w'v\in {Q^\perp} w_0$ in which case they are equal.
\\ If 
$v w\in {Q^\perp} w_0$ and $w'v\in {Q^\perp} w_0$ then $\ell(vw)= \ell(v)+1$, $\ell(w'v)=\ell(v)+1$. So  $\theta(T_v T_w)=\theta(T_{vw})=\theta(T_{w'v})=\theta(T_{w'}T_v)$.
Otherwise $v w\not\in {Q^\perp} w_0$ and $w'v\notin {Q^\perp} w_0$. 
\begin{itemize}
\item If $\ell(vw)= \ell(v)+1$ and $\ell(w'v)=\ell(v)+1$ then $\theta(T_v T_w)=\theta(T_{w'}T_v)=0$.
\item If $\ell(vw)= \ell(v)+1$ and $\ell(w'v)=\ell(v)-1$,  we still have  $\theta(T_v T_w)=0$ and 
the quadratic relations implies that $T_{w'} T_v$ is a linear combination of $T_{w'v}$ and $T_{\omega' v}$. None of $w'v$ and $\omega' v$ lie in ${Q^\perp} w_0$. So again $\theta(T_{w'}T_v)=0$.
\item If $\ell(vw)= \ell(v)-1$ and $\ell(w'v)=\ell(v)+1$ then $\theta(T_{w'}T_v)=0$, as above we find  $\theta(T_v T_s)=\theta(T_{s'}T_v)=0$.
\item Suppose that $\ell(vw)= \ell(v)-1$ and $\ell(w'v)=\ell(v)-1$. If $v\omega\in {Q^\perp} w_0$ then $\omega' v=v\omega$ and (by \eqref{quad}) we have: 
$\theta(T_v T_w)=\theta((\a+\b) T_{v\omega})=(\a+\b)\theta(T_{\omega'v})=\theta(T_{w'} T_v)$.
If $v\omega\notin {Q^\perp} w_0$, then $\omega' v\notin{Q^\perp} w_0$ and again,  $\theta(T_v T_w)=\theta(T_{w'} T_v)=0$.
\end{itemize}
\end{proof}
Recall from Proposition \ref{prop:frees} that $\HF$ is finitely generated over $\z$ which is a finitely generated $k$-algebra. Hence $\HF$ has a ridig dualizing complex.
\begin{corollary}  \label{cor:HF}
The $\z$-rigid dualizing complex $\R_{\HF/\z}$ of $\HF$ is isomorphic to $(i_F) \HF$.  In particular,  
$$\R_{\HF} \cong \R_\z \otimes_\z (i_F) \HF$$ 
and $\HF$ is Gorenstein with the same self-injective dimension as $\z$.
\end{corollary}
\begin{proof}  
Recall that $\z$ is regular. By Proposition \ref{prop:finite/Z} and Lemma \ref{lemma:HF} it follows that $\R_{\HF/z} \cong (i_F) \HF$. Lemma \ref{lem:dualizing} then implies that $\R_{\HF} \cong \R_\z \otimes_\z (i_F) \HF$. 

By Proposition \ref{prop:finite/Z} and Corollary \ref{cor:dual1} the injective dimension of $\R_{\HF}$ (and hence of $\HF$) is at most the injective dimension of $\R_\z$ (which is the same as the self-injective dimension of $\z$). But since $\HF$ is free over $\z$ (by Proposition \ref{prop:freesF}) this inequality must be an equality.  Thus $\HF$ is Gorenstein with the same self-injective dimension as $\z$. 
\end{proof}

\subsection{The rigid dualizing complex of $\Hh$} \label{subsec:omegaH}

\begin{proposition}\label{prop:Hdiff}
The algebra $\Hh$ is a differential $\k$-algebra of finite type.
\end{proposition}
\begin{proof}
One can filter $\Hh$ by length so that $\r[\a,\b]$ lies in the smallest filtered piece. Then the associated graded algebra is isomorphic to $H_{0,0} \otimes_\r \r[\a,\b]$. By Proposition \ref{prop:finite} this is finite over its center which is a finitely generated $k$-algebra.  
\end{proof}

As a consequence of Proposition \ref{prop:Hdiff} and the discussion in Section \ref{sec:existence} it follows that $\Hh$ has a rigid dualizing complex. In this section we will identify this complex explicitly. 

\begin{theorem}\label{theo:main}
We have $\R_{\Hh/\z} \cong (\upiota) \Hh[d] \in D^b(\Hh \otimes_\z H_{\a,\b}^o)$ where $d = \rk(Q)$ and $\upiota$ is defined in \eqref{f:upiota}. 
\end{theorem}

Since $\z$ is regular, this has the following consequence by combining with Lemma \ref{lem:dualizing}.

\begin{corollary}  \label{coro:main} 
We have $\R_{\Hh} \cong (\upiota) \Hh[d] \otimes _\z \R_\z$.
\end{corollary}

The rest of this section is devoted to proving Theorem \ref{theo:main}. To simplify notation we will write $\Hn$ instead of $\Hh$. Recall the resolution 
\begin{equation}\label{f:shortreso} \Hn_d \to \dots \to \Hn_0  \to \Hn  \end{equation}
of $\Hn \otimes_\z \Hn^o$-modules from \eqref{f:complexH} where $\Hn_i := \bigoplus_{F \in \mathscr F_i} \Hn(j_F)\otimes_{\HF} \Hn$.  We will check below that:
\begin{enumerate}
\item $\Hn$ is Gorenstein (Lemma \ref{lem:HGor}),
\item each $\RHom_{\Hn \otimes_\z  \Hn^o}(\Hn_i, \Hn \otimes_\z  \Hn)$ is supported in cohomological degree zero (Lemma \ref{lemma:onedeg}),
\item the cokernel of the induced map 
$$\RHom_{\Hn \otimes_\z  \Hn^o}(\Hn_{d-1}, \Hn \otimes_\z  \Hn) \to \RHom_{\Hn \otimes_\z  \Hn^o}(\Hn_d, \Hn \otimes_\z  \Hn)$$
is isomorphic to $(\upiota) \Hn$ as an $\Hn$-bimodule (Lemma \ref{lem2}).
\end{enumerate}
By (1) and Lemma \ref{lem:dualizing}, \cite[Thm. 3.9]{Ye1} we know that the $\z$-rigid dualizing complex $\R_{\Hn/\z}$ is supported in one cohomological degree. By Proposition \ref{prop:gorenstein-dualizing} we also know that $\R_{\Hn/\z}^{-1} \cong \RHom_{\Hn \otimes_\z \Hn^o}(\Hn, \Hn \otimes_\z \Hn)$. Finally, from (2) and (3) we conclude that this is isomorphic to $(\upiota) \Hn [-d]$, from which Theorem \ref{theo:main} follows. 

\begin{lemma}\label{lem:HGor}
The algebra $\Hn$ is Gorenstein with self-injective dimension $\rk(X)$. 
\end{lemma}
\begin{proof} 
Tensoring (\ref{f:shortreso}) with an $\Hn$-module $M$, we get a length $d$ resolution of $M$ by $\Hn$-modules of the form $\Hn(j_F) \otimes_{\HF} M$. On the other hand, by adjunction we have 
$$\RHom_{\Hn}(\Hn(j_F) \otimes_{\HF} M, \Hn) \cong \RHom_{\HF}(M|_{\HF}, \Hn|_\HF).$$
Since $\Hn$ is a free $\HF$-module $\Hn|_{\HF}$ and $\HF$ have the same (finite) injective dimension. It follows that $\Hn$ has finite self-injective dimension.   Lastly, $\Hn$ is Noetherian by Corollary \ref{coro:noeth}.

The argument above actually implies that the self-injective dimension of $\Hn$ is bounded above by $d+r$ where $d=\rk(Q)$ and $r$ is the self-injective dimension of $\z$. Given Corollary \ref{coro:main}, an application of Proposition \ref{prop:fin-module} implies this bound is sharp, namely $\Hn$ has injective dimension $d+r=\rk(X)$. 

\end{proof}
\begin{remark} Compare Lemma \ref{lem:HGor} with \cite[Theorems 0.1, 0.2-ii]{OS1}.
\end{remark}

\begin{lemma}\label{lemma:onedeg}  
The object $\RHom_{\Hhz}(\Hn_i, \Hz)$  is supported in  cohomological  degree zero.
\end{lemma}
\begin{proof}
For $F \in \mathscr F_i$ we want to show that 
$$\RHom_{\Hhz}(\Hn(j_F)\otimes_{\HF} \Hn, \Hz)$$ 
is supported in degree zero. First note that, by base change
$$\Hn(j_F) \otimes_{\HF} \Hn \cong (\Hn(j_F) \otimes_{\z} \Hn) \otimes_{\HF \otimes_\z \HF} \HF \cong (\Hn \otimes_{\z} \Hn) \otimes_{\HF \otimes_\z \HF} (j_F) \HF$$
where, for the second isomorphism, we use  the  fact that $j_F$ is an involution. Thus we get 
\begin{align*}
\RHom_{\Hhz}(\Hn(j_F)\otimes_{\HF} \Hn, \Hz) 
&\cong \RHom_{\Hhz}((\Hhz)\otimes_{{\HFz}} (j_F)\HF, \Hz)\cr
&\cong \RHom_{{\HFz}}( (j_F)\HF, \Hz)\end{align*}
where the second isomorphism is by adjunction. Since $\Hn$ is a free $\HF$-module  (Proposition \ref{prop:frees}), it follows that $\Hn \otimes_\z \Hn$ is a free $\HFz$-module. Thus, since $(j_F)\HF$ is a finitely presented $\HFz$-module, it suffices to show that $\RHom_{{\HFz}}( (j_F)\HF, \HF \otimes_\z \HF)$ (or equivalently  $\RHom_{{\HFz}}(\HF, \HF \otimes_\z \HF)$) is supported in degree zero.  
Since $\HF$ is Gorenstein (Corollary \ref{cor:HF}) and $\z$ regular, it follows from Proposition \ref{prop:gorenstein-dualizing} that   $\RHom_{{\HFz}}(\HF, \HF \otimes_\z \HF)$ is isomorphic to $\R_{\HF/\z}^{-1} \cong (i_F^{-1}) \HF$ which is clearly supported in degree zero.  
\end{proof} 

Next we apply $\RHom_{\Hhz}(-, \Hz)$ to \eqref{f:shortreso} and study the cohomology on the far right. 

\begin{lemma} \label{lem2}
The cokernel of 
$$\RHom_{\Hhz}(H_{d-1},\Hz) \to \RHom_{\Hhz}(H_d, \Hz)$$
is isomorphic to $(\upiota) H$ as an $\Hhz$-module.
\end{lemma}
\begin{proof} 
By Lemma \ref{lemma:onedeg} we need to identify the cokernel of the map
\begin{equation}\label{partial*} 
\partial^*:   \Hom_{\Hhz}(\bigoplus_{F\in \mathscr F_{d-1}}\Hn(j_F)\otimes_{\HF} \Hn,\Hz)  \to \Hom_{\Hhz} (\Hn(j_C)\otimes_{\HC} \Hn, \Hz)
\end{equation}
induced by the following map from \eqref{f:diffC}
$$\Hn(j_C)\otimes_{\HC} \Hn \ni   \quad 1 \otimes 1 \longmapsto  \sum_{F\in \in \mathscr F_{d-1}} \sum_{\omega\in \Omega/\Omega_{F}} j_C(T_\omega) \otimes T_{\omega^{-1}} \quad \in \bigoplus_{F\in \mathscr F_{d-1}}\Hn(j_F)\otimes_{\HF} \Hn.$$
The modules $\Hz$ inside the Homs above carry two actions of $\Hhz$ that we will need to keep track of. For   $x\otimes y \in \Hn\otimes_{\z} \Hn$ we have:
\begin{itemize}
\item[-] the \emph{outer} action $(T \otimes S) \ast (x \otimes y) := Tx \otimes yS$  (where $T\otimes S\in \Hhz$),
\item[-] the \emph{inner} action $(x\otimes y) \star  (T \otimes S) := xS\otimes Ty$ (where $T\otimes S\in \Hhzo$).
\end{itemize}
Note that, to simplify notation, we write the inner action as a right action of $\Hn^{o}\otimes_{\z} \Hn$. Since the inner and outer actions commute the spaces $ \RHom_{\Hhz}(H_i, \Hz)$ are $\Hhz$-modules via the inner action. 

\medskip

\noindent For a facet $F$ contained in  $\overline C$, we denote by $\mathfrak M_F$ the subspace $$\mathfrak M_F =\{X\in \Hn\otimes_{\z} \Hn, \: (j_F(T_w)\otimes 1-1\otimes T _w)\ast X= 0\:\:\forall w\in \WF\}  .$$ 
It is a right  submodule of $\Hz$ over $\Hhzo$ (for the $\star$ action).
Using  base change as in the proof of Lemma \ref{lemma:onedeg}, 
the left hand space in \eqref{partial*} identifies with 
$$\bigoplus_{F\in \mathscr F_{d-1}}
\Hom_{{\HFz}}( (j_F)\HF, \Hz)\cong \bigoplus_{F\in \mathscr F_{d-1}}
\mathfrak M_F $$ and the right hand side with
$$
\Hom_{{\HCz}}( (j_C)\HC, \Hz)\cong
\mathfrak M_C  \ .$$
Hence we are studying the cokernel of 
the map
\begin{equation}\label{concretedelta}\begin{array}{crcl} {\partial}^*:& \bigoplus_{F\in \mathscr F_{d-1}}
\mathfrak M_F&\longrightarrow& \mathfrak M_C \cr 
&X\in \mathfrak M_F&\longmapsto &\sum_{\omega\in \Omega/\Omega_{F}} j_C(T_\omega) \, \otimes\, T_{\omega^{-1}}\ast X \ .\end{array}\end{equation} It is an homomorphism of right $\Hhzo$-modules. For $F$ a facet contained in  $\overline C$, 
we define $$\theta_F:=\sum_{\omega\in \Omega_F/{Q^\perp}} j_F(T_\omega) \, \otimes\, T_{\omega^{-1}}=\sum_{\omega\in \Omega_F/{Q^\perp}} \epsilon_F(\omega) T_\omega \, \otimes\, T_{\omega^{-1}}\in \Hz \ $$ which we will  also see  as an element in $\Hhz$ or $\Hhzo$.
\begin{fact}\label{fac1}  We have $\mathfrak M_C=\theta_C\ast \Hz$. It coincides  with the right $\Hhzo$-module  generated by $\theta_C$.

\end{fact}

\begin{proof}[Proof of fact \ref{fac1}]
Since $H$ is a free $\r[\Omega]$-module on the left and on the right, the left  $\HCz$-module $\Hz$ (for the $\ast$ action) is a direct sum of copies of $\HCzm$. It is therefore enough to check the equality
\begin{equation}\label{int}\{X\in \HCzm, \: (j_C(T_\omega)\otimes 1-1\otimes T _\omega)\ast X= 0\:\:\forall \omega\in \Omega\} =\theta_C\ast(\HCzm)\ .\end{equation}
For the indirect inclusion, it is easy to check that the product
$(j_C(T _{\omega})\otimes 1-1\otimes T _{\omega})\theta_C$ in $\HCz$ is zero for $\omega\in\Omega$.
For the direct inclusion, we fix a system of representatives $U$ of $\Omega/{Q^\perp}$  containing $1\in\Omega$ 
and we consider the basis  $\{T_u\otimes T_v\}_{u\in U, v\in \Omega}$ of $\HCzm$.
Consider a generic element
\begin{equation}X=\sum_{u\in U, v\in \Omega} \lambda_{u,v} T_u\otimes T_v\quad \in \HCzm\ .\label{xgene}\end{equation} Assume it lies in the left hand space of \eqref{int}.
Let $\omega\in U$ and $v\in \Omega$. Considering the  component in  $T_1\otimes T_v$ of
$(j_C(T_{\omega^{-1}})\otimes 1-1\otimes T _{\omega^{-1}})\ast X$
we see that  $\epsilon_C(\omega^{-1})\lambda_{\omega,v}=\lambda_{1, \omega v}$ and therefore
$$X=\sum_{u\in U, v\in \Omega} \epsilon_C(u) \lambda_{1,uv} T_u\otimes T_v= \sum_{u\in U, v\in \Omega} \epsilon_C(u) \lambda_{1,v} T_u\otimes T_{vu^{-1}}=\theta_C\ast
\sum_{v\in \Omega}  \lambda_{1,v} T_1\otimes T_{v}\ .
$$\end{proof}

For $F$ a facet of codimension $1$ of $C$, we denote by $s_F\in \Saff$ the reflexion defining the wall containing $F$.

\begin{fact}\label{fac2} For $F$ a facet of codimension $1$ of $C$,  we have $\mathfrak M_F=\theta_F(T _{s_F}\otimes 1-1\otimes \upiota(T _{s_F})\ast \Hz$ which coincides with the right $\Hhzo$-module generated by 
$\theta_F\ast(T _{s_F}\otimes 1-1\otimes \upiota (T _{s_F}))=
(T _{s_F}\otimes 1-1\otimes \upiota (T _{s_F}))\ast \theta_F$. 
\end{fact}

\begin{proof}[Proof of fact \ref{fac2}]   
Recall that conjugation by an element in $\Omega_F$ leaves $s_F$ invariant. Therefore $\HF$ is a commutative algebra and $\theta_F$ and  $(T _{s_F}\otimes 1-1\otimes \upiota (T _{s_F}))$
commute in $\HFz$.
Using Proposition \ref{prop:frees}, the left  $\HFz$-module $\Hz$ is a direct sum of copies of $\HFzm$. Therefore, it is enough to show that
$$\{X\in \HFzm, \: (j_F(T _{w})\otimes 1-1\otimes T _w)\ast X= 0\:\:\forall w\in W_F^\dagger\} =(T _{s_F}\otimes 1-1\otimes \upiota (T _{s_F}))\theta_F\ast(\HFzm)\ .$$

\noindent 
For the indirect inclusion,  first recall that 
  $\upiota (T _{s_F})=\iota(T _{s_F})=\alpha- T _{s_F}$ and $T _{s_F}
\upiota (T _{s_F})=-\beta$ where momentarily we set $\alpha:=\a+\b$ and $\beta:=-\a\b$.
 It is enough to verify, in $\HFz$:
\begin{align*}(j_F(T _{s_F})\otimes 1-1\otimes T _{s_F})(T _{s_F}\otimes 1-1\otimes \upiota (T _{s_F}))&=(T _{s_F}\otimes 1-1\otimes T _{s_F})(T _{s_F}\otimes 1-1\otimes \upiota (T _{s_F}))\cr&=(\alpha T _{s_F}+\beta)\otimes 1-T _{s_F}\otimes \upiota (T _{s_F})-T _{s_F}\otimes T _{s_F}-\beta\otimes 1\cr&=
\alpha T _{s_F}\otimes 1-T _{s_F}\otimes (\alpha- T _{s_F})-T _{s_F}\otimes T _{s_F}=0\end{align*}
and  the  identity
$(j_F(T _{\omega})\otimes 1-1\otimes T _{\omega})\theta_F=0$
for $\omega\in\Omega_F$.

\medskip

\noindent For the direct inclusion, notice that the (commutative) algebra
$\HFz$ is a tensor product of $H_F\otimes_\r H_F$ by $\r[\Omega_F]\otimes_{\z} \r[\Omega_F]$. Therefore, it is enough to show that \\1) an element in
$H_F\otimes_\r H_F$ which is annihilated by $(j_F(T _{s_F})\otimes 1-1\otimes T _{s_F})\ast$ lies in  $(T _{s_F}\otimes 1-1\otimes \upiota (T _{s_F}))\ast H_F\otimes_\r H_F$,\\
2) 
 an element in
$\r[\Omega_F]\otimes_{\z} \r[\Omega_F]$ which is annihilated by
$(j_F(T _{\omega})\otimes 1-1\otimes T _{\omega})\ast $ for all $\omega\in \Omega$ lies in   $\theta_F\ast(\r[\Omega_F]\otimes_{\z} \r[\Omega_F])$. \\
For 2) we proceed just like in the proof of Fact \ref{fac1}. For 1),   we consider a generic element
$$X=\lambda_{1,1} T_1\otimes T_1+
\lambda_{1,s} T_1\otimes T_{s_F}+
\lambda_{s,1} T_{s_F}\otimes T_{1}+\lambda_{s,s} T_{s_F}\otimes T_{s_F}\in H_F\otimes_\r H_F\ .$$
Assuming that $X$ is annihilated by $(j_F(T _{s_F})\otimes 1-1\otimes T _{s_F})$, we obtain by direct calculation that 
$$X=(T _{s_F}\otimes 1-1\otimes \upiota (T _{s_F}))\ast(\lambda_{s,s} T_1\otimes T_{s_F}+ \lambda_{1,s} T_1\otimes T_1)\ .$$

\end{proof}
\begin{fact} \label{fac3}The  image of the map $\partial^*$ as in \eqref{concretedelta} is 
the right $\Hhzo$-module generated by all $ \theta_C \ast (1 \otimes h - \upiota (h) \otimes 1)$ for $h\in H$.

\end{fact}

\begin{proof}[Proof of Fact \ref{fac3}] For the direct inclusion,  let $F\in \mathscr F_{d-1}$.
By Fact \ref{fac2}, the right $\Hhzo$-module $\mathfrak M_F$ is generated by
$\theta_F\ast (1\otimes T_{s_F}-\upiota (T_{s_F})\otimes 1)$ whose image by $\partial^*$
is $\theta_C\ast (1\otimes T_{s_F}-\upiota (T_{s_F})\otimes 1).$ We use here the fact that $\epsilon_C\vert {\Omega_F}= \epsilon_F$ because the choice of an orientation on $C$ determines an orientation on its codimension $1$ facets.\\\\
For the indirect inclusion, we show, for $w\in W$, that
$\theta_C\ast(1\otimes T_w-\upiota (T_w)\otimes 1)\text{  lies in the image of $\partial ^*$.}$
We proceed by induction on $\ell(w)$.\\
1)  If $\ell(w)=0$ namely $w\in \Omega$, then $\upiota (T_w)=\epsilon_C(w)T_w$ and 
$\theta_C\ast(1\otimes T_w-\upiota (T_w)\otimes 1)=0$.\\
2) If $\ell(w)=1$, we first consider the case $w= s_F$ for $F$ a facet in $\mathscr F_{d-1}$.
Then $\theta_F\ast (1\otimes T_{s_F}-\upiota (T_{s_F})\otimes 1)$ lies in $\mathfrak M_F$ and according to \eqref{concretedelta}  its image by $\partial ^*$ is 
$\theta_C\ast (1\otimes T_{s_F}-\upiota (T_{s_F})\otimes 1)$. \\ For $s\in \Saff$, there is $\omega\in \Omega$ and $F\in \mathscr F_{d-1}$  such that $s=\omega s_F \omega^{-1}$.
Then using the identity $\theta_C(T_{\omega}\otimes T_{\omega^{-1}})= \epsilon_C(\omega)\theta_C$, we verify
$
\theta_C\ast (1\otimes T_{s}-\upiota (T_{s})\otimes 1) = \epsilon_C(\omega) \theta_C\ast (1\otimes T_{ s_F} -\upiota ( T_{s_F} )\otimes 1)\star  T_{\omega}\otimes T_{ \omega^{-1}} \in \im(\partial^*).
$\\
Lastly, write an arbitrary element of length $1$  as $ \omega s$ for $s\in \Saff$ and $\omega \in \Omega$. We have
$\theta_C\ast (1\otimes T_{\omega s}-\upiota (T_{\omega s})\otimes 1) = \theta_C\ast (1\otimes T_{s}-\upiota (T_{s})\otimes 1)\star  T_\omega\otimes 1 \in \im(\partial^*)$ (using $(j_C(T _{\omega})\otimes 1-1\otimes T _{\omega})\theta_C=0$).
\\
3) Now let $w$ of length $\geq 1$ and  $s\in \Saff$ such that $\ell(sw)=\ell(w)+1$. We have
 \begin{align*}
\theta_C\ast (1\otimes T_{ s} T_w-\upiota (T_{s} T_w)\otimes 1)=
\theta_C\ast (1\otimes T_w-\upiota ( T_w)\otimes 1)\star T_s\otimes 1+\theta_C\ast (1\otimes T_s-\upiota (T_{s} )\otimes 1)\star 1\otimes \upiota ( T_w)
\end{align*}
which lies in $\im(\partial^*)$ by induction.

\end{proof}

The $\Hhz$-module $(\upiota)  H$ is equivalently a right $\Hhzo$-module with action 
$(h, T\otimes S)\mapsto \upiota (T)h S\ .$
The surjective map
\begin{equation} \begin{array}{cccc}\mu:& \Hz&\longrightarrow&(\upiota)  \Hn\cr &x\otimes y&\longmapsto & \upiota (y) x\end{array}\end{equation}
is then equivariant for the $\Hhzo$-action on the right. 
  
\begin{fact} The map $\mu$ factors through the kernel of the (left)  action of $\theta_C\,\ast\, $.

 \end{fact}
\begin{proof} 
In a first step we let $X=\sum_{u\in U, v\in \Omega} \lambda_{u,v} T_u\otimes T_v$ be a generic element in  $\HCzm$ as in \eqref{xgene}, where  $U$ is a chosen set of representatives of $\Omega/{Q^\perp}$ containing $1$.
 Write $\theta_C=\sum_{\omega\in U}\epsilon_C(\omega) T_ {\omega^{-1}}\otimes T_\omega$.
Given $v\in \Omega$, the coefficient of the  component in $T_1\otimes T_v$ of $\theta_C\ast X$ is  $\sum_{\omega\in U} \lambda_{\omega, v\omega^{-1}}\epsilon_C(\omega)$ while $\mu(X)=\sum_{v\in \Omega} \epsilon_C(v)T_v \sum_{\omega\in U}\lambda_{\omega,v \omega^{-1}} \epsilon_C( \omega)$. So 
$\theta_C\ast X=0$ implies  $\mu(X)=0$.

\noindent Now notice that $\Hz$ is a free left $\HCzm$-module (for the $\ast$ action) with basis $T_x\otimes T_y$, $x,y\in \Waff$.
An element  $X\in \Hz$  may therefore be written as $X=\sum_{ x,y\in \Waff} X_{x,y} \ast T_x\otimes T_y=\sum_{ x,y\in \Waff} X_{x,y} \star T_y\otimes T_x$ with $X_{x,y}\in \HCzm$. Assume
$\theta_C \ast X=0$. It  is equivalent to $\theta_C \ast X_{x,y}=0$  which implies $\mu(X_{x,y})=0$ for all $x,y\in \Waff$. But  $\mu$ is right $\Hhzo$-equivariant so $\mu(X)=0$.
\end{proof}

By Fact \ref{fac3}, it is clear that  $\mu$ also factors through $\im(\partial ^*)$.
Hence we have a well defined surjective right $\Hhzo$-equivariant map on $ \coker(\partial^*)=\mathfrak M_C/\im(\partial ^*)=(\theta_C\ast \Hz)/ \im(\partial ^*)$ given by
 \begin{equation} \begin{array}{clcc}\bar \mu:& \coker(\partial^*)&\longrightarrow& (\upiota) \Hn\cr &\theta_C\ast X \bmod \im(\partial ^*)&\longmapsto & \mu(X) \end{array}\end{equation} for $X\in \Hz$.
 We show that it is  also  injective by introducing the linear map
 $$f: (\upiota) \Hn\longrightarrow \coker(\partial^*), \quad  h\longmapsto \theta_C\ast (h\otimes 1) \bmod   \im(\partial ^*) \ .$$
Notice, for  $x\otimes y\in \Hz$ that \begin{align*}f(\mu( x\otimes y))-\theta_C\ast (x\otimes y)&=\theta_C\ast  (\upiota (y)x\otimes 1-x\otimes y)=\theta_C\ast ( \upiota (y)\otimes -1\otimes y)\star 1\otimes x\in \im(\partial^*)\end{align*}
This shows that $f\circ \bar\mu=\id_{\coker(\partial^*)}$
and $\bar\mu$ is bijective.
 \end{proof}
 
\section{The structure of $\Hh$ over its center}\label{sec:frob}

In this section we continue to assume $\r$ is a regular, finitely generated $k$-algebra.

\subsection{Projectivity over the center}

Let us denote by $\eZ_{\q^\pm}, \eZ_0$ and $\eZ_{\a,\b}$ the centers of $H_{\q^\pm}, H_0$ and $\Hh$ respectively. 

\begin{lemma}\label{lem:center}
We have $\eZ_{\q^\pm} \cong \r[\q^\pm][\cX]^{W_0}$ and $\eZ_0 \cong \r[\cX]^{W_0}$ with $H_{\q^\pm}$ and $H_0$ finitely generated over them.  
\end{lemma}
\begin{proof}
If $\q$ is invertible then classical results going back to Bernstein (\cite[Prop. 3.11]{Lu}) show that there exists a commutative subalgebra $A_{\q^\pm} \subset H_{\q^\pm}$ such that 
\begin{enumerate}
\item $H_{\q^\pm}$ is finitely generated and free over $A_{\q^\pm}$, 
\item $A_{\q^\pm} \cong \r[\q^\pm][\cX]$ and 
\item $\eZ_{\q^\pm} := A^{W_0}_{\q ^\pm} \cong \r[\q^\pm][\cX]^{W_0}$ is the center of $H_{\q^\pm}$.
\end{enumerate}
This proves the claim involving $H_{\q^\pm}$. 

The picture for $H_0$ is similar. The algebra $H_0$ again contains a commutative subalgebra $A_0$ over which it is finite \cite[Thm 3]{Vigann}). Moreover, $\eZ_0 \cong A_0^{W_0}$ (\cite[Thm 4]{Vigann}, see also \cite[2.2.3]{Ollcompa}). One then follows the proof of  \cite[Prop. 2.10]{Ollcompa} to identify $\eZ_0$ with $\r[\cX^+]$ which is isomorphic to $\r[\cX]^{W_0}$ (cf. Theorem in Section 2.4 of \cite{Lo}).
\end{proof}

Lemma \ref{lem:center} suggests the following natural extension.

\begin{conjecture}\label{conj:center}
We have $\eZ_{\a,\b} \cong \r[\a,\b][\cX]^{W_0}$ with $\Hh$ finitely generated over it. 
\end{conjecture}

Next we recall the following properties involving $\Z[\cX]^{W_0}$. If $X/Q$ is free then
\begin{itemize}
\item $\Z[\cX]$ is free over $\Z[\cX]^{W_0}$,
\item $\Z[\cX]^{W_0} \cong \Z[Q^\perp] \otimes_\Z \Z[\cX/Q^\perp]^{W_0}$,
\item $\Z[\cX/Q^\perp]^{W_0}$ is a polynomial algebra in the fundamental coweights. 
\end{itemize}
The first result follows from the Pittie-Steinberg theorem \cite[Thm. 1.1]{St2}. More precisely, the first statement above is proven in \cite[Thm. 2.2]{St2} (with $X$ in place of $\cX$).  The second and third results follow from \cite[Thm. 1.2(c)]{St2} (interpreted in terms of root systems) and \cite[Thm. 6.1]{St1}. 

In particular, if $X/Q$ is free, $\r[\cX]^{W_0}$ is isomorphic to the tensor product of a Laurent series ring and a polynomial ring. For example, for the based root systems associated to ${\rm PGL}_n$, ${\rm GL}_n$ and ${\rm SL}_n$ the quotients $X/Q$ are  trivial, $\Z$ and $\Z/n\Z$ respectively. Meanwhile, $\Z[\cX]^{W_0}$ is isomorphic to $\Z[x_1,\dots,x_{n-1}]$, $\Z[x_1,\dots,x_{n-1},x_n^\pm]$ and $\Z[x_1,\dots,x_{n-1}]^{\Z/n\Z}$ respectively (where the $\Z/n\Z$ acts by $x_i \mapsto x_{i+1}$ for $i \ne n-1$ and $x_{n-1} \mapsto -(x_1+\dots+x_{n-1})$). 

\begin{proposition}\label{prop:projective}
If $X/Q$ is free then $H_{\q^\pm}$ and $H_0$ are both projective over their centers. 
\end{proposition}
\begin{proof}
If $X/Q$ is free then $\eZ_0 \cong \r[\cX]^{W_0}$ is regular, connected. On the other hand, by Corollary \ref{coro:main}, $\R_{\Hh}$ is supported in one degree. By base change  (Proposition \ref{prop:basechange}) this implies that $\R_{H_0}$ is also supported in one degree.  Thus Proposition \ref{prop:miracle} (miracle flatness) implies that $H_0$ is  projective over $\eZ_0$. The proof for $H_{\q^\pm}$ is the same. 
\end{proof}
\begin{remark}\label{rem:sl2}
If the based root system is associated to the group $SL_2$ then $X/Q \cong \Z/2\Z$ is not free. However, the results of Proposition \ref{prop:projective} still hold since in this case $\r[\cX]^{W_0} \cong \r[x]^{\Z/2\Z} \cong \r[x^2]$ is a polynomial algebra so the argument from Proposition \ref{prop:projective} still applies. This recovers \cite[Cor. 3.4]{embed}. 
\end{remark}
\begin{remark}\label{rem:free}
In the case of $H_{\q^\pm}$ the standard proof that it is projective over its center uses the fact that it is free over the intermediate subalgebra $A_{\q^\pm}$. However, this direct argument fails in the case of $H_0$. In this case the algebra $A_0$ is more complicated and, in particular, $H_0$ is not projective over $A_0$ (see the introduction of \cite{GL3}). Nevertheless, it is still true (if $X/Q$ is free) that $H_0$ is projective over $\eZ_0$ as a result of miracle flatness (as used in the proof of Proposition \ref{prop:projective}). 
\end{remark}

\begin{proposition}\label{prop:RH/Z}
If $X/Q$ is free then $\R_{H_{\q^\pm}/\eZ_{\q^\pm}} \cong  (\upiota) H_{\q^\pm}$ and $\R_{H_0/\eZ_0} \cong (\upiota) H_0$. 
\end{proposition}
\begin{proof}
We prove the claim for $H_0$ as the proof for $H_{\q^\pm}$ is the same. By Proposition \ref{prop:finite/Z} 
$$\R_{H_0} \cong \RHom_{\eZ_0}(H_0, \eZ_0) \otimes_{\eZ_0} \R_{\eZ_0}$$
where $\eZ_0$ is regular because $X/Q$ is free. On the other hand, since $H_0 = \Hh/(\a,\b-1)$, we obtain from Corollary \ref{coro:main} using base change that
$$\R_{H_0} \cong (\upiota) H_0 [d] \otimes_{\z_{0,1}} \R_{\z_{0,1}}$$
where $\z_{0,1} := \z/(\a,\b-1) \cong \r[Q^\perp]$. It follows that 
$$\RHom_{\eZ_0}(H_0, \eZ_0) \cong  (\upiota) H_0 [d] \otimes_{\eZ_0} \R_{\eZ_0}^{-1} \otimes_{\r[Q^\perp]} \R_{\r[Q^\perp]} \cong  (\upiota) H_0 [d] \otimes_{\eZ_0} \R_{\eZ_0/\r[Q^\perp]}^{-1} \cong  (\upiota) H_0.$$
The third isomorphism above is because $\eZ_0 \cong \r[Q^\perp] \otimes_\r \r[\cX/Q^\perp]^{W_0}$ where $\r[\cX/Q^\perp]^{W_0}$ is a polynomial algebra in $d$ variables over $\r$ and therefore $\R_{\eZ_0/\r[Q^\perp]}\cong Z_0[d]$. 
\end{proof}

\begin{corollary}\label{cor:conj}
If $X/Q$ is free then assuming Conjecture \ref{conj:center} the algebra $\Hh$ is finitely generated, projective over $\eZ_{\a,\b}$ and $\R_{\Hh/\eZ_{\a,\b}} \cong (\upiota) \Hh$. 
\end{corollary}
\begin{remark}\label{rem:Hh}
The result above follows as in the proofs of Propositions \ref{prop:projective} and \ref{prop:RH/Z}. Conversely, applying base change (cf. Corollary \ref{prop:basechange}),  Corollary \ref{cor:conj} immediately implies Propositions \ref{prop:projective} and \ref{prop:RH/Z}. 
\end{remark}

\begin{remark}
By Proposition \ref{prop:finite/Z}, we know that $\R_{H_0} \cong \RHom_{Z_0}(H_0, \R_{Z_0})$ which implies in particular that $\R_{H_0}$ is an $H_0\otimes_{Z_0}  H_0^o$-module.  As we saw above, using base change, we have
$\R_{H_0} \cong (\upiota) H_0 [d] \otimes_{\z_{0,1}} \R_{\z_{0,1}}$. It follows that $\upiota$ must act trivially on $Z_0$. 
The same is true for $H_{\q^\pm}$ and $Z_{\q^{\pm}}$. This recovers 
 \cite[Prop. 3.2]{Ollcompa} which was proved by explicit calculation.

\end{remark}
\subsection{Frobenius structure}

Consider a pair of rings $S \subset S'$. Recall that $S'$ is a Frobenius (resp. free Frobenius) extension of $S$ if:
\begin{enumerate}
\item $S'$ is a finitely generated, projective  (resp. free) $S$-module,
\item there exists an isomorphism $\phi: S' \xrightarrow{\sim} \Hom_S(S',S)$ of $(S',S$)-bimodules.
\end{enumerate}
In such cases one can define the bilinear form $S' \times S' \to S$ via $\la x,y \ra := \phi(y)(x)$. 
Suppose $S \subset S'$ is a free Frobenius extension with $S$ commutative. By \cite[Cor. 1.2]{BF}, if $\{x_i\}$ is a basis of $S'$ over $S$ then the matrix $[\la x_i, x_j \ra]_{i,j}$ is an invertible matrix over $S$.

\begin{corollary}\label{cor:frob}
Suppose $X/Q$ is free and $\r = \k$. Then $H_{\q^\pm}$ and $H_0$ are both free Frobenius extensions over their centers. The same is true of $\Hh$ if we assume Conjecture \ref{conj:center}. The Nakayama automorphism in all three cases is $\upiota$. 
\end{corollary}
\begin{proof}
By Lemma \ref{lem:center} and Proposition \ref{prop:projective}, $H_0$ is finitely generated, projective over $\eZ_0$. Since this center is the tensor products of a polynomial and Laurent series and $\r=\k$ it follows (e.g. \cite[Thm. 2.1]{G}) that $H_0$ are actually free over $\eZ_0$. The free Frobenius structure now follows since, as we saw in Proposition \ref{prop:RH/Z}, $\RHom_{\eZ_0}(H_0, \eZ_0) \cong  (\upiota) H_0$. The proofs involving $H_{\q^\pm}$ and $\Hh$ are the same. 
\end{proof}

\end{document}